\documentclass[DIV=12,dvipsnames]{scrartcl}

\usepackage{mathtools}
\usepackage{amsmath,amssymb,mathrsfs,amsfonts,bm}

\usepackage{graphicx}
\usepackage{tcolorbox}
\graphicspath{{figures/}}

\usepackage{url}
\usepackage{xurl}

\usepackage{tabularx}
\usepackage{booktabs}
\usepackage{enumitem}

\usepackage{amsthm}

\theoremstyle{remark}
\newtheorem{remark}{Remark}
\newtheorem{example}{Example}

\theoremstyle{definition}
\newtheorem{definition}{Definition}

\theoremstyle{plain}
\newtheorem{lemma}{Lemma}
\newtheorem{theorem}{Theorem}
\newtheorem{proposition}{Proposition}
\newtheorem{corollary}{Corollary}

\newcommand{\funding}[1]{#1}
\newcommand{\email}[1]{\href{mailto:#1}{#1}}
\usepackage{scrlayer-scrpage}
\newcommand{\headers}[2]{
  \clearpairofpagestyles
  \lehead{\pagemark}\cehead{#2}
  \rohead{\pagemark}\cohead{#1}

}
\usepackage[colorlinks,urlcolor=RoyalBlue,citecolor=Green,linkcolor=BrickRed]{hyperref}
\usepackage[capitalise]{cleveref}

\newcommand{\wh}{\ensuremath{\widehat}}
\newcommand{\DA}{\ensuremath{\varDelta A}}
\newcommand{\DB}{\ensuremath{\varDelta B}}
\newcommand{\DC}{\ensuremath{\varDelta C}}

\newcommand{\C}{\ensuremath{\mathbb{C}}}
\newcommand{\R}{\ensuremath{\mathbb{R}}}

\newcommand{\Tr}{\operatorname{Tr}}

\DeclareMathOperator{\opvec}{vec}
\DeclareMathOperator{\diag}{diag}
\DeclareMathOperator{\sep}{sep}
\DeclarePairedDelimiter{\abs}{\lvert}{\rvert}
\DeclarePairedDelimiter{\norm}{\lVert}{\rVert}
\DeclareMathOperator{\rank}{rank}

\newcommand{\upxtwo}{\ensuremath{\overline{\gamma}_X}}

\newcommand{\ignore}[1]{}

\newcommand{\thetitle}{Conditioning of solutions\\to the Sylvester equation}
\newcommand{\thefunding}{This work of was supported by the Engineering and Physical Sciences Research Council [grant numbers UKRI2774; UKRI4034]}

\newcommand{\theauthori}{Massimiliano Fasi}
\newcommand{\theaffiliationi}{School of Computer Science, University of Leeds, Woodhouse Lane, Leeds
  LS2 9JT, UK}
\newcommand{\theemaili}{m.fasi@leeds.ac.uk}

\newcommand{\theauthorii}{Behnam Hashemi}
\newcommand{\theaffiliationii}{School of Computing and Mathematical Sciences, University of Leicester, University Road, Leicester LE1 7RH,
UK}
\newcommand{\theemailii}{b.hashemi@leicester.ac.uk}

\title{\thetitle\thanks{Version of 2 September 2026. \funding{\thefunding}.}}
\headers{Conditioning of solutions to the Sylvester equation}{M. Fasi and B. Hashemi}
\author{
  \theauthori{}\thanks{\theaffiliationi{} (\email{\theemaili{}}).}
  \and
  \theauthorii{}\thanks{\theaffiliationii{} (\email{\theemailii{}}).}}

\date{}

\begin{document}

\maketitle

\begin{abstract}
We partially answer an open problem, posed by Nick Higham, concerning the conditioning of solutions to Sylvester and Lyapunov equations. The question arises in the backward stability analysis of numerical algorithms for these equations. We first show that the solution to the Sylvester equation $AX-XB = C$ can be arbitrarily ill-conditioned even if $A, B, C$ and the Kronecker sum $I \otimes A - B^T \otimes I$ are all perfectly conditioned. We then derive general a priori bounds on the condition number of the solution, as well as bounds for the Sylvester equation when $A$ and $B$ are diagonalizable. We also provide lower bounds involving matrix exponentials and Zolotarev numbers. For the Lyapunov equation $AX+XA^T = -C$, we obtain upper bounds in two settings: (i) when $A$ is symmetric positive definite while $C$ is symmetric negative definite, and (ii) when $A$ is strictly dissipative and $C$ is symmetric positive definite.
\end{abstract}



\section{The Sylvester and Lyapunov equations}

Consider the Sylvester matrix equation
\begin{equation}
\label{sylv:eq}
A X - X B = C
\end{equation}
where $A\in\R^{m\times m}$, $B\in\R^{n\times n}$, and $C\in\R^{m\times n}$ are given, and $X\in\mathbb{R}^{m\times n}$ is the solution to be determined.
We assume without loss of generality that $m \ge n$.
It is well-known that~\eqref{sylv:eq} has a unique solution if and only if $A$ and $B$ do not have a common eigenvalue.

How should we measure the accuracy of a computed solution $\wh X$ to~\eqref{sylv:eq}?
A natural answer is to consider the \emph{forward error}, which measures the distance between $\wh X$ and the exact solution $X$ to the original equation.
The forward error depends on both the sensitivity of the problem being solved and on the algorithm used to solve it, and it may be large even for a stable algorithm, if the problem is ill conditioned.

To isolate these two factors, one can rely on the \emph{backward error}, which asks for the smallest perturbations $\DA$, $\DB$, $\DC$, in a suitable normwise or componentwise sense, such that $\wh X$ is the exact solution to the perturbed Sylvester equation
\begin{equation*}
    (A + \DA) \wh X - \wh X (B + \DB) = C + \DC.
\end{equation*}

A third quantity of interest is the \emph{residual}
\begin{equation*}
    R = A\wh X - \wh X B - C,
\end{equation*}
which will be small if $\wh X$ nearly satisfies the original equation.

In the context of linear systems, a small residual is equivalent, after a suitable normalization, to a small backward error, and a computed solution with a small residual can be interpreted as the exact solution of a nearby linear system.
This is not necessarily the case for Sylvester equations: Higham~\cite{high93y} shows that a small residual does not imply a small backward error, but the backward error can be large if $\wh X$ is ill conditioned.

This motivates our interest in deriving conditions under which the Sylvester equation has a well-conditioned solution.
Higham identifies this as an open problem~\cite[p.~311]{high02} and states the problem explicitly in the analysis of the backward stability of a computed approximate solution $\wh X$ to~\eqref{sylv:eq} \cite[p.~319]{high02}.
The backward error bounds used to arrive at this condition rely on the quantity
\begin{equation}\label{eq:kappa-Y}
    \kappa_2(\wh X) = \norm{\wh X^{+}}_2 \norm{\wh X}_2 = \frac{\sigma_1(\wh X)}{\sigma_r(\wh X)},
\end{equation}
where $r \le n$ is the rank of $\wh X$, and $\sigma_1(\wh X)$ and $\sigma_r(\wh X)$ are the largest and smallest nonzero singular values of $\wh X$, respectively.

Higham~\cite[p.~312]{high02} applies a similar reasoning to the Lyapunov equation
\begin{equation}
    \label{lyap:eq-intro}
    AX + XA^T = C,
\end{equation}
noting that the condition number of the solution $X$ plays a role in the analysis of the backward error of~\eqref{lyap:eq-intro}.

We stress that the conditioning of $X$, as defined in~\eqref{eq:kappa-Y}, differs from the sensitivity of the Sylvester equation as a linear operator.
In fact, the Sylvester equation~\eqref{sylv:eq} is linear and can be recast as a linear system of order $mn$.
Let $\otimes$ denote the infix Kronecker operator and let $\operatorname{vec}$ denote the operator that stacks the columns of a matrix into a length-$mn$ vector. Vectorizing~\eqref{sylv:eq} yields the linear system
\begin{equation}
\label{sylv:eq-kron}
\mathcal{A} x = c,\qquad
\mathcal{A} = I_n\otimes A - B^{T}\otimes I_m \in \R^{mn \times mn},\qquad
x=\opvec(X) \in \R^{mn},
\end{equation}
where $I_k \in \R^{k \times k}$ is the identity matrix of order $k$.

Solving ~\eqref{sylv:eq} is equivalent to solving~\eqref{sylv:eq-kron}, and the spectral condition number of $\mathcal{A}$ is
\begin{equation*}
    \kappa(\mathcal{A}) = \norm{\mathcal{A}}_2\norm{\mathcal{A}^{-1}}_2 = \frac{\sigma_{\max}(\mathcal{A})}{\sigma_{\min}(\mathcal{A})},
\end{equation*}
which is an upper bound on, but in general not equal to, the structured condition number $\Psi$ of the Sylvester equation~\cite[p.~313]{high02}; related structured condition numbers are developed by Ghavimi and Laub~\cite{ghla95}.

In order to study the conditioning of~\eqref{sylv:eq}, Varah~\cite{vara79} uses the \emph{separation}, an operator originally introduced by Stewart~\cite{stew73} to study invariant subspaces and defined as
\begin{equation}\label{eq:separation}
    \sep(A, B) = \min_{X \neq 0} \frac{\norm{AX - XB}_F}{\norm{X}_F}.
\end{equation}
Since $\sep(A, B) = \sigma_{\min}(\mathcal{A})$, the separation quantifies the sensitivity of the solution to perturbations in $C$ when $A$ and $B$ are fixed.

The subject of our work is the conditioning of $X$ itself, which depends not only on the separation of the spectra of $A$ and $B$, but also on how the right-hand side $C$ interacts with the singular directions of the operator $\mathcal{A}$; see \Cref{experiments-sep-cond:sec} for an example.
We identify situations in which the solution~$X$ is guaranteed to be well conditioned, and we show that $X$ can be ill conditioned even when $\mathcal{A}$ is well conditioned.

Existing literature has mostly focused on conditions under which the solution $X$ is nonsingular.
Luenberger~\cite{luen65} considers the operator equation $TA-BT = C$ on a Banach space.
Assuming that the spectra of $A$ and $B$ are disjoint and that $C$ has one-dimensional range, Luenberger gives necessary and sufficient conditions for the unique solution $T$ to have trivial nullspace and dense range.
In finite dimensions, these conditions are equivalent to $T$ being of full rank.
Hearon~\cite{hear77} considers~\eqref{sylv:eq} with $m = n$ and $C = xy^*$, for $x, y \in \C^n$ and shows that if the equation is consistent and the pairs $(B, x)$ and $(A^*, y)$ are controllable, then every solution $X$ is nonsingular. This sufficient condition is also necessary if the equation has a unique solution.
Datta~\cite{datt88} gives necessary and sufficient conditions for a solution to~\eqref{sylv:eq} to be nonsingular, under the assumptions that $m = n$ and that $A$ is a normalized lower Hessenberg matrix. These conditions are based on a determinantal inequality that only involves the first $n-1$ rows of $C$.

To the best of our knowledge, only Higham~\cite{high93y} and Ghavimi and Laub~\cite{ghla95} discuss the backward error of the Sylvester equations. Higham~\cite{high93y} shows that the conventional relative residual can underestimate the backward error by an arbitrarily large factor, derives a sharp structured perturbation bound whose associated condition number can be arbitrarily smaller than the usual bound based on $\sep(A, B)^{-1}$, and develops a practically computable a posteriori bound on the forward error. Ghavimi and Laub~\cite{ghla95} extend these ideas to nonsymmetric and symmetric algebraic Riccati equations, with Sylvester and Lyapunov equations as special cases. For the Sylvester equation, they give explicit formulae for the optimal perturbations in the Frobenius norm, and they introduce a structured relative residual that is equivalent, up to a modest constant factor, to the corresponding backward error. They also propose procedures to computed an estimate of the condition number and refine it iteratively. This body of work also derives a structured expression for the backward error and a sharp perturbation bound with an associated condition number that respects the Kronecker structure of the problem. Our focus is the conditioning of the solution matrix~$X$ in the 2-norm.

The remaining sections are organized as follows. \Cref{sec:preliminaries} collects some general results used to derive bounds in following sections. \Cref{sec:special-cases} considers the simplified setting in which either $A$ or $B$ is zero. \Cref{sec:general-case} presents examples demonstrating that, in general, $X$ can be arbitrarily ill-conditioned even when $A$, $B$, $C$,  and $\mathcal{A}$ are all well conditioned. \Cref{sec:general-bounds} derives general a priori bounds on $\kappa_F(X)$ and $\kappa_2(X)$. For diagonalizable $A$ and $B$, \cref{sec:diagonalizable} establishes bounds on $\kappa_2(X)$, including lower bounds in terms of Zolotarev numbers, as well as results for the Lyapunov equation~\eqref{lyap:eq-intro}. \Cref{integral-rep:sec} uses integral representations to derive further lower and upper bounds for $\kappa_2(X)$ in the Lyapunov case. Finally, \cref{sec:experiments} concludes the discussion by presenting some illustrative numerical experiments.

\section{Preliminaries}\label{sec:preliminaries}

\begin{lemma} \label{basic-sv-ineq:lem}
(i) Let $M$ and $N$ be conformable matrices, at least one of which is square. Then, \cite[p.~616, Cor.~9.6.6]{Bernstein09}
\begin{align}
\sigma_{\min} (M)\, \sigma_{\min} (N)
\le
\sigma_{\min} (MN)
\le&
\sigma_{\max} (M)\, \sigma_{\min} (N)\label{min-sv-prod-bounds:eq}\\
\sigma_{\min} (M)\, \sigma_{\max} (N)
\le
\sigma_{\max} (MN)
\le&
\sigma_{\max} (M)\, \sigma_{\max} (N)\label{max-sv-prod-bounds:eq}
\end{align}
\ignore{
(i) Let $M$ and $N$ be matrices of compatible size such that $M$ has full column rank and $N \neq 0$ or $M \neq 0$ and $N$ is of full row rank. Then,
\[
\sigma_{\min} (MN) \ge \sigma_{\min} (M) \sigma_{\min} (N)
\]
where $\sigma_{\min}(X)$ refers to the smallest NONZERO singular value of $X$.
}

(ii) Let $M$ and $N$ be two $m \times n$ matrices. Then, \cite[p.~617, Cor.~9.6.9]{Bernstein09}
\begin{align}
\sigma_{\min} (M) - \sigma_{\max} (N) \le& \sigma_{\min} (M \pm N) \le \sigma_{\min} (M) + \sigma_{\max} (N).\label{sum-minsv-bound:eq}
\end{align}

\end{lemma}

Note that part (i) does not necessarily hold if both matrices $M$ and $N$ are rectangular.

\begin{definition}
Let $M$ be a square matrix. The logarithmic norm induced by the Euclidian norm, denoted by $\mu(M)$, is the largest eigenvalue of its symmetric part, i.e.,
\[
\mu(M) = \lambda_{\max} \biggl(\frac{M+M^T}{2}\biggr).
\]
\end{definition}

The following result bounds the matrix exponential in terms of the logarithmic norm. It is based on~\cite[Thm.~10.11]{high08}  and $\mu(t M) = t \mu(M)$ for any $t \ge 0$.
\begin{lemma} \label{log-norm:lem}
For any square matrix $M$ and any $t \ge 0$, we have $\| e^{M t} \|_2 \leq e^{t\mu(M)}$.
\end{lemma}

\begin{definition}
A square matrix $M$ is called strictly dissipative if its symmetric part $\frac{M+M^T}{2}$ is negative definite.
\end{definition}
If $M$ is strictly dissipative, then there exist constants $0< \beta \le \gamma$ such that the eigenvalues of the symmetric part of $M$ lie in the real interval $[-\gamma, -\beta]$ and this is equivalent to
\[
\mu(M) \le -\beta, \qquad \mbox{ and } \qquad \mu(-M) \le \gamma.
\]
Note also that while every symmetric negative definite matrix is dissipative, a strictly dissipative matrix does not have to be symmetric.

We now state a result that we will use later. We use the operator $\sep_2$, which denotes the separation defined as in~\eqref{eq:separation}, but using the 2-norm instead of the Frobenius norm.

\begin{lemma}\label{sol-log-sep:lem}
Let $A$ and $B$ be real square matrices. Then,
\begin{equation}\label{eq:spectral-sep}
\sep_2(A,B) \ge -(\mu(-A) + \mu(B)).
\end{equation}
\end{lemma}
\begin{proof}
    If $\mu(-A) + \mu(B) \ge 0$, then~\eqref{eq:spectral-sep} is satisfied because $\sep(A,B)_2 \ge 0$.
    Otherwise, let $u$ and $v$ be the unit left and right singular vectors of $X$ corresponding to the leading singular value $\norm{X}_2$. Then, from
    \begin{equation*}
        Xv = \norm{X}_2 u,\qquad X^Tu = \norm{X}_2 v,\qquad \norm{u}_2 = \norm{v}_2 = 1,
    \end{equation*}
    we can write
    \begin{equation*}
        \norm{AX-XB}_2 \ge \abs{u^T(AX - XB)v} = \norm{X}_2 \abs{u^TAu - v^TBv} \ge -(\mu(-A) + \mu(B))\norm{X}_2,
    \end{equation*}
    where the last step follows from the fact that $u^TAu \ge -\mu(-A)$ and $v^TBv \le \mu(B)$. Dividing by $\norm{X}_2$ and minimizing over $X \neq 0$ concludes the proof.
\end{proof}

\section{Simpler special cases}
\label{sec:special-cases}
We begin by the case in which $A$ is nonsingular and $B=0$. Then~\eqref{sylv:eq} reduces to
\begin{equation}
\label{rightInverse:eq}
A X = C,
\end{equation}
whose solution is $X = A^{-1} C$ and included in this case is (right) matrix inversion. Using a consistent
norm, we obtain the bounds
\begin{align*}
\|X\| &= \|A^{-1}C\| \le \|A^{-1}\|\,\|C\|,\\
\|C\| &= \|A X\| \le \|A\|\,\|X\|,
\end{align*}
which yield
\begin{equation}
\label{rightInverse-bound:eq}
\frac{\|C\|}{\|A\|} \le \|X\| \le \|A^{-1}\| \|C\|
\end{equation}
and in particular for the 2-norm
\[
\frac{\|C\|_2}{\sigma_{\max}(A)} \le \|X\|_2 \le \frac{\|C\|_2}{\sigma_{\min}(A)}.
\]

If $X$ is invertible (so $C$ is invertible as well), then $X^{-1}=C^{-1}A$, and similarly, we obtain
\begin{align*}
& \|X^{-1}\|  = \|C^{-1}A\| \le \|C^{-1}\|\,\|A\|,\\
& \|A\|  =\|C X^{-1}\| \le \|C\| \, \|X^{-1}\|
\quad\Rightarrow\quad
\|X^{-1}\| \ge \frac{\|A\|}{\|C\|}.
\end{align*}
Thus one may write (in analogy with the bounds for $\|X\|$
\begin{equation}
\label{rightInverse-inverse-bound:eq}
\frac{\|A\|}{\|C\|} \le \|X^{-1}\| \le \|C^{-1}\| \|A\|,
\end{equation}
and therefore for the two norm
\[
\frac{\|A\|_2}{\sigma_{\max}(C)} \le \|X^{-1}\|_2 \le \frac{\|A\|_2}{\sigma_{\min}(C)}.
\]
Combining the bounds~\eqref{rightInverse-bound:eq} and~\eqref{rightInverse-inverse-bound:eq} gives a bound on the condition number $\kappa(X)$.
\begin{equation}
\label{kappa-AXB:eq}
1 \le \kappa(X) \le \kappa(A)\,\kappa(C).
\end{equation}

Hence, in the special case $B=0$, the conditioning of the solution $X$ is controlled by the conditioning of $A$ and $C$: if both $A$ and $C$ are well-conditioned, then $X$ will be well-conditioned; conversely,  $X$ could be ill conditioned, if $A$ or $C$ have a large condition number.

If either $A$ or $C$ is ill-conditioned, then the solution $X$ to~\eqref{rightInverse:eq} might be ill-conditioned. For instance,
\begin{itemize}
\item Let $A=I$ and let $C=H_n$ be the ill conditioned Hilbert matrix. Then $X=C$ is as ill-conditioned.
\item Let $C=I$ and let $A=H_n$ be the Hilbert matrix. Then $X=A^{-1}$ is as ill conditioned as~$A$.
\end{itemize}

In the same spirit, it is straightforward to see that the solution of $X B = C$ satisfies
\[
1 \le \kappa(X) \le \kappa(B)\,\kappa(C).
\]

\ignore{
\subsection{Special case $A=0$, $B$ nonsingular}
While~\eqref{rightInverse:eq} corresponds to computing a right-inverse of $A$, the special case
\begin{equation}
\label{leftInverse:eq}
X B = C,
\end{equation}
corresponds to computing a left inverse of $B$. Although left and right inverses of a nonsingular matrix coincide in exact arithmetic, they may behave quite differently in finite precision arithmetic. For an illustrative contrast, see the image on the front cover of Higham's book~\cite{high02}; see also~\cite[p.~262]{high02}. We also note that the MATLAB routine \texttt{inv} computes a left inverse~\cite{rump2006error}, which helps explain why the left and right residuals may differ substantially in practice.

It is easy to see that the solution to~\eqref{leftInverse:eq} satisfies
\[
1 \le \kappa(X) \le \kappa(B)\,\kappa(C).
\]
}

\section{General case} \label{sec:general-case}
Unfortunately, the bounds derived for the special cases in \cref{sec:special-cases} do not generalize to the Sylvester equation.
To see this, consider~\eqref{sylv:eq} and the following example.
\begin{example} \label{amazing:ex}
Let
\begin{equation}\label{eq:counterexample}
  A = \begin{pmatrix} 0 & 1\\1 & 0 \end{pmatrix},\qquad
  B = \begin{pmatrix} 0 & -2\\-2 & 0 \end{pmatrix},\qquad
  X = \begin{pmatrix} \varepsilon & 0\\0 & 1 \end{pmatrix},
\end{equation}
The eigenvalues of $A$ and $B$ are $\pm 1$ and $\pm 2$, respectively, thus the equation has a unique solution for any $C$. For the solution $X$ in \eqref{eq:counterexample} we have
\begin{equation*}
  C = \begin{pmatrix} 0 & 1+2\varepsilon\\2 + \varepsilon & 0 \end{pmatrix}.
\end{equation*}
The three matrices $A$, $B$, and $C$ are all well conditioned, since
\begin{equation*}
  \kappa_{2}(A) = \kappa_{2}(B) = 1,\qquad 1 < \kappa_{2}(C) = \frac{2+\varepsilon}{1 + 2\varepsilon} \le 2,
\end{equation*}
but the solution can be arbitrarily ill conditioned, since $\kappa_2(X) = \varepsilon^{-1}$ tends to $\infty$ as $\varepsilon$ nears~0.

This example also shows that the solution to a well-conditioned Sylvester equation can be ill conditioned.
In fact, the Kronecker matrix $\mathcal{A}$ in \eqref{sylv:eq-kron} has two singular values equal to $1$ and two equal to $3$, thus $\kappa_2(\mathcal{A}) = 3$ and the equation is well conditioned irrespective of $C$.

\begin{example}
\label{eq:counterexample-lyap}
We note that the same phenomenon occurs for the Lyapunov equation $AX+XA^T = -C$ with $A$, $X$ and $C$ as follows
\begin{equation}\label{eq:Lyap-counterexample}
  A = \begin{pmatrix} 2 & 1\\0 & -1 \end{pmatrix},\qquad
  X = \begin{pmatrix} \varepsilon & 0\\0 & 1 \end{pmatrix},\qquad
C = \begin{pmatrix} -4\varepsilon & -1\\-1 & 2 \end{pmatrix}
\end{equation}
for which $\kappa_2(A) \approx 2.6$, $\kappa_2(\mathcal{A}) \approx 5.7$, and $\kappa_2(C) < 5.9$ for $\varepsilon$ near zero.
\end{example}

It is worth observing that, in both \cref{amazing:ex,eq:counterexample-lyap}, all the input matrices $A$, $B$ and $C$ are indefinite.
\end{example}
The main strength of the bounds we derive in \cref{sec:general-bounds} lies in their applicability to general problems. The price of this generality, however, is that the bounds can be rather loose for examples such as the two above. In \cref{sec:diagonalizable,integral-rep:sec}, we therefore turn to bounds tailored to more specific classes of problems. In particular, in \cref{spd-A-C-lyap:thm,Lyap-dissipativeA:thm}, we establish sufficient conditions that guarantee that bounds resembling~\eqref{kappa-AXB:eq} hold for the Lyapunov equation.

\section{General a priori bounds on the condition number} \label{sec:general-bounds}

\paragraph{Frobenius norm} We begin by deriving a lower bound on the condition number in the Frobenius norm.
Let $X$ have full rank, and let $\sigma_1(X) \ge \ldots \ge \sigma_n(X) > 0$ be its singular values.
Then
\begin{equation}
    \norm{X}_F^2=\sum_{i=1}^n \sigma_i(X)^2,
    \qquad
    \norm{X^+}_F^2=\sum_{i=1}^n \sigma_i(X)^{-2}.
\end{equation}
Hence, by Cauchy's inequality, we have that
\begin{equation*}
\kappa^2_F(X) = \norm{X}_F^2\norm{X^+}_F^2
=
\left(\sum_{i=1}^n \sigma_i(X)^2\right)
\left(\sum_{i=1}^n \sigma_i(X)^{-2}\right)
\ge
\left(\sum_{i=1}^n 1\right)^2
=
n^2,
\end{equation*}
and taking the square root gives the lower bound $X^+$
\begin{equation*}
\norm{X^+}_F \ge \frac{n}{\norm{X}_F}.
\end{equation*}

As \cref{amazing:ex} suggests, we cannot give a nontrivial upper bound purely in terms of $A$, $B$, and $C$. However, we can bound the Frobenius norm condition number from above in terms of the spectral condition number. Since
\begin{equation*}
    \kappa_2(x) = \frac{\sigma_1(X)}{\sigma_n(X)}
\end{equation*}
Since $\sigma_n^2 \le \sigma_i(X)^2 \le \sigma_1(X)^2$, Kantorovich's inequality~\cite[p.~63, Fact 1.17.37]{Bernstein09}
gives
\begin{equation*}
\kappa_F^2(X)=
\left(\sum_{i=1}^n \sigma_i(X)^2\right)
\left(\sum_{i=1}^n \sigma_i(X)^{-2}\right)
\le
n^2 \frac{\bigl(\sigma_1(X)^2+\sigma_n(X)^2\bigr)^2}{4\sigma_1(X)^2\sigma_n(X)^2},
\end{equation*}
and by taking the square root we obtain
\begin{equation}\label{eq:cond-frob-upper-bound}
\kappa_F(X)
\le n \frac{\sigma_1(X)^2+\sigma_n(X)^2}{2\sigma_1(X)\sigma_n(X)}
= \frac{n}{2}\left(\kappa_2(X) + \frac{1}{\kappa_2(X)}\right).
\end{equation}
Therefore, any upper bound on $\kappa_2(X)$ yields an upper bound on $\kappa_F(X)$.

We remark that~\eqref{eq:cond-frob-upper-bound} is tighter than the bound $\kappa_F(X) \le n \kappa_2(X)$, which follows directly from the norm inequality $\norm{A}_F \le \sqrt{\rank(A)} \norm{A}_2$.

\paragraph{Spectral norm} In order to bound $\kappa_2(X)$, we need bounds on the 2-norm of $X$ and $X^+$.
We begin by bounding these in terms of the Frobenius norm.

To bound $\norm{X}_2$, note that
\begin{equation*}
    \norm{C}_F = \norm{\opvec C}_2 = \norm{\mathcal{A} \opvec (X)}_2 \le \sigma_{\max}(\mathcal A)\norm{X}_F,
\end{equation*}
which combined with the well-known matrix inequality
\begin{equation*}
    \norm{X}_2 \ge \frac{\norm{X}_F}{\sqrt{n}}
\end{equation*}
yields
\begin{equation}
    \label{eq:lower-sep-X}
    \frac{\norm{C}_F}{\sqrt{n}\sigma_{\max}(\mathcal{A})}
    \le
    \norm{X}_2
    \le
    \norm{X}_F.
\end{equation}

We now turn to $\norm{X^+}_2$.
For the lower bound, from
\begin{equation*}
    \norm{X}_F \le \norm{\mathcal{A}^{-1}}_2 \norm{C}_F = \frac{\norm{C}_F}{\sigma_{\min}(\mathcal{A})}
\end{equation*}
we can conclude that
\begin{equation}
    \label{eq:lower-sep-Xpsinv}
    \norm{X^+}_2 = \frac{1}{\sigma_{\min}(X)}
    \ge \frac{1}{\norm{X}_F}
    \ge \frac{\sigma_{\min}(\mathcal{A})}{\norm{C}_F}
    = \frac{\sep(A,B)}{\norm{C}_F}.
\end{equation}

Deriving an upper bound on $\norm{X^+}$ requires more care. Note that
\[
\lVert x\rVert_F = \lVert x\rVert_2 = \lVert X\rVert_F,
\]
but the analogous equality does not hold for the pseudoinverse.
In fact, as long as $X \neq 0$, we have that
\begin{equation}
\label{eq:vec-norm}
\opvec(X)^+ = x^+ = \frac{x^T}{x^Tx} \neq \opvec(X^+),
\end{equation}
where $x^+$ is a row vector while $\opvec(X^+)$ is a column vector.
Thus,
\[
\lVert x^{+}\rVert_2 = \lVert x^{+}\rVert_F \neq \lVert X^{+}\rVert_F.
\]
A consequence of~\eqref{eq:vec-norm} is that
\begin{equation*}
    \norm{x^+}_2 = \frac{\sqrt{x^Tx}}{x^T x} = \frac{1}{\norm{x}_2},
\end{equation*}
and
\[
\|x^{+}\|_{2} \le \frac{1}{\sigma_{\min}(X)} = \|X^{+}\|_{2}.
\]

If $m\ge n$ and $X$ has full column rank, the pseudoinverse \mbox{$X^{+}=(X^{T}X)^{-1}X^{T}$} satisfies \mbox{$X^{+}X=I_n$}. Consequently, the spectral norm of \(X^{+}\) is the reciprocal of the smallest singular value of \(X\):
\begin{equation}\label{eq:tnorm-ps-inv}
\lVert X^{+}\rVert_2 = \frac{1}{\sigma_{\min}(X)} = \frac{1}{\sigma_n(X)}.
\end{equation}
The Frobenius norm of \(X^{+}\) can be expressed as
\begin{equation}\label{eq:fnorm-ps-inv}
\lVert X^{+}\rVert_F^2 = \sum_{i=1}^n \frac{1}{\sigma_i(X)^2}.
\end{equation}
From~\eqref{eq:tnorm-ps-inv} and \eqref{eq:fnorm-ps-inv}, we have
\begin{equation*}
\norm{X^{+}}_F = \left(\sum_{i=1}^n \frac{1}{\sigma_i(X)^2}\right)^{1/2} \ge
    \sqrt{\frac{1}{\sigma_n(X)^2}} = \frac{1}{\sigma_n(X)} = \norm{X^{+}}_2,
\end{equation*}
and using the fact that $\sigma_{\min}(X) \le \norm{X}_F$ gives the lower bound
\begin{equation}\label{eq:upper-Xpsinv}
    \norm{X^+}_2 = \frac{1}{\sigma_{\min}(X)} \ge \frac{1}{\norm{X}_F}.
\end{equation}
Combining~\eqref{eq:lower-sep-Xpsinv} and~\eqref{eq:upper-Xpsinv}, we obtain the chain of inequalities
\begin{equation}\label{eq:bounds-norm-psinv}
    \frac{\sep(A,B)}{\norm{C}_F}
    \le \frac{1}{\norm{X}_F}
    \le \norm{X^{+}}_2
    \le \norm{X^{+}}_F.
\end{equation}

Combining \eqref{eq:lower-sep-X} and \eqref{eq:lower-sep-Xpsinv} yields the lower bound
\begin{equation}
\label{eq:norm2-lower-bound}
    \kappa_2(X) \ge \frac{1}{\sqrt{n}}\cdot\frac{1}{\kappa(\mathcal{A})}.
\end{equation}
We note that~\eqref{eq:norm2-lower-bound} is not very informative, because it is can never be stronger than the obvious $\kappa_2(A) \ge 1$.

We can obtain a sharper lower bound in terms the spectral norm of $A$ and $B$, as we now explain.
To bound $\norm{X}_2$ from above, observe that for any norm we have
\[
\norm{C}=\norm{AX-XB} \ge \Big| \norm{AX}-\norm{XB} \Big|,
\]
and that for the spectral norm we also have\footnote{Also, when $m=n$, we have $\|AX\|_2 \ge \sigma_{\min}(X)\,\|A\|_2$. It follows that
\[
\kappa_2(X) \ge \frac{\|A\|_2}{\|B\|_2} - \frac{\|C\|_2}{\|B\|_2\ \sigma_{\min} (X)}
\]}
\begin{equation}\label{eq:norm-to-sv}
\norm{AX}_2 \ge \sigma_{\min}(A)\norm{X}_2,\qquad
\norm{XB}_2 \le \norm{X}_2\norm{B}_2.
\end{equation}
If $\sigma_{\min}(A) > \norm{B}_2$, then
\[
\norm{C}_2 \ge \norm{AX}_2-\norm{XB}_2 \ge \norm{X}_2 \big(\sigma_{\min}(A)-\norm{B}_2\big) \ge 0
\]
gives the upper bound
\begin{equation}
\label{X-2-norm-upper-bound:eq}
\norm{X}_2 \le \frac{\norm{C}_2}{\sigma_{\min}(A)-\norm{B}_2}.
\end{equation}
Analogously, we can swap the role of $A$ and $B$ in~\eqref{eq:norm-to-sv} to obtain
\begin{equation}
\label{X-2-norm-upper-bound-B:eq}
\norm{X}_2 \le \frac{\norm{C}_2}{\sigma_{\min}(B)-\norm{A}_2},
\end{equation}
assuming that $\sigma_{\min}(B) > \norm{A}_2$, so that the denominator of~\eqref{X-2-norm-upper-bound-B:eq} is positive.

\begin{remark}
The bound~\eqref{X-2-norm-upper-bound:eq} also appears in~\cite[p.~18]{NakatsukasaThesis}, where it is used to give an elegant proof of the $\sin\Theta$ theorem~\cite[pp.~180--189]{NakatsukasaThesis}.
\end{remark}

We can also bound $\norm{X}_2$ from below in terms of the 2-norms of $A$, $B$, and $C$.
In fact, from
\begin{equation*}
    \norm{C}_2 = \norm{AX - XB}_2 \le \bigl(\norm{A}_2 + \norm{B}_2\bigr)\norm{X}_2
\end{equation*}
we immediately obtain
\begin{equation}
    \label{X-2-norm-lower-bound:eq}
    \norm{X}_2 \ge \frac{\norm{C}_2}{\norm{A}_2 + \norm{B}_2}.
\end{equation}
Combining~\eqref{X-2-norm-upper-bound:eq}, \eqref{X-2-norm-upper-bound-B:eq}, and~\eqref{X-2-norm-lower-bound:eq} we obtain the chain of inequalities
\begin{equation}\label{eq:bounds-norm-X}
    \frac{\norm{C}_2}{\norm{A}_2 + \norm{B}_2} \le \norm{X}_2 \le \upxtwo,\qquad
    \upxtwo = \frac{\norm{C}_2}{\max\{\sigma_{\min}(A)-\norm{B}_2, \sigma_{\min}(B)-\norm{A}_2\}}.
\end{equation}

From~\eqref{eq:bounds-norm-psinv} and~\eqref{eq:bounds-norm-X}, we obtain
\begin{equation}
    \label{eq:alternative}
    \begin{aligned}
    \kappa_2(X) &= \norm{X}_2 \norm{X^+}_2
    \ge \frac{\norm{C}_2}{\norm{A}_2 + \norm{B}_2} \frac{\sep(A, B)}{\norm{C}_F}
    \ge \frac{\sep(A, B)}{\sqrt{\rank C}\bigl(\norm{A}_2 + \norm{B}_2\bigr)},
    \end{aligned}
\end{equation}
using the fact that $\norm{C}_F \le \sqrt{\rank C} \cdot \norm{C}_2$.

Is is straightforward to obtain an upper bound on the spectral condition number in terms of the Frobenius condition number, since
\begin{equation*}
    \kappa_2(X) = \norm{X}_2 \norm{X^{+}}_2 \le \norm{X}_F \norm{X^{+}}_F = \kappa_F(X).
\end{equation*}

\section{Bounds for diagonalizable matrices}
\label{sec:diagonalizable}

The derivation in this section closely follows that in~\cite[Sec.~3.1]{frha12}. Assume that $A$ and $B$ are both
diagonalizable, so that they we have the spectral decompositions
\begin{align} \label{specdecA:eq}
A &= V_A D_A V_A^{-1},\quad
&&\text{with}\quad V_A, D_A \in \mathbb{C}^{m \times m},\quad
&D_A &= \diag(\lambda_1,\ldots,\lambda_m),\\
 \label{specdecB:eq}
B &= V_B D_B V_B^{-1},\quad
&&\text{with}\quad V_B, D_B \in \mathbb{C}^{n \times n},\quad
&D_B &= \diag(\mu_1,\ldots,\mu_n).
\end{align}
Thus, columns $i$ of $V_A$ and $V_B$ are eigenvectors of $A$ and $B$
with eigenvalue $\lambda_i$ and $\mu_i$, respectively. We then have
\[
D_A (V_A^{-1} X V_B) - (V_A^{-1} X V_B) D_B = V_A^{-1} C V_B,
\]
and by defining
\begin{align}
\widetilde X &= V_A^{-1} X V_B, \label{Xtilde-def:eq}\\
\widetilde C &= V_A^{-1} C V_B, \label{Ctilde-def:eq}
\end{align}
we can reformulate~\eqref{sylv:eq} as the Sylvester equation
\begin{equation}
\label{sylv-diagonal:eq}
D_A \widetilde X - \widetilde X D_B = \widetilde C,
\end{equation}
which has diagonal coefficients. In particular,
\[
(D_A \widetilde X - \widetilde X D_B)_{ij} = (\lambda_i - \mu_j) \widetilde X_{ij},
\]
and as long as the spectra of $A$ and $B$ have an empty intersection, the elements of $\widetilde X$ can
\begin{equation} \label{sylv-diagonal-entries:eq}
\widetilde X_{ij} = \frac{\widetilde C_{ij}}{\lambda_i - \mu_j}.
\end{equation}

\ignore{Let $\lambda$ and $\mu$ denote the $m \times 1$ and the $n \times 1$ vectors containing the eigenvalues of $A$ and $B$, respectively. Let also $e_k$ denote the $k \times 1$ vector of all-ones. In addition, we define the full $m \times n$ matrix $D$ as
\[
D = \lambda e_n^T - e_m^T \mu
\]
and $R = 1\oslash D$ as the entywise division of $1$ by entries of $D$. More precisely, as}

Therefore, we can write the matrix $\widetilde X$ as the Hadamard product
\[
\widetilde X = \widetilde C \circ R,
\]
where $R$ is the Cauchy matrix with entries
\begin{equation}\label{rij:eq}
R_{ij} = \frac{1}{\lambda_i - \mu_j}.
\end{equation}
Finally,
\begin{equation}
\label{X-spectral-dec:eq}
X = V_A \big( \widetilde C \circ R \big) V_B^{-1}.
\end{equation}
Using part (i) of \cref{basic-sv-ineq:lem}, we have therefore proved the following result.

\begin{proposition}
\label{diagonalizable-sylv:prop}
Let $A\in\mathbb \R^{m\times m}$ and $B\in\mathbb \R^{n\times n}$ be diagonalizable, with eigendecompositions~\eqref{specdecA:eq} and \eqref{specdecB:eq}, respectively.
Assume that $\lambda_i\neq \mu_j$ for all $i,j$, and define $\widetilde C$ and $\widetilde R$ as in~\eqref{Ctilde-def:eq} and~\eqref{rij:eq}, respectively. Then the unique solution to~\eqref{sylv:eq} is
\begin{equation*}
X=V_A(\widetilde C\circ R)V_B^{-1},
\end{equation*}
where $\circ$ denotes the Hadamard product. If $X$ has full rank, then
\begin{equation}\label{eq:statement}
\frac{\kappa_2(\widetilde C\circ R)}
{\kappa_2(V_A)\kappa_2(V_B)}
\le
\kappa_2(X)
\le
\kappa_2(V_A)\kappa_2(V_B)
\kappa_2(\widetilde C\circ R).
\end{equation}
\end{proposition}

\begin{proof}
By~\eqref{min-sv-prod-bounds:eq} and~\eqref{max-sv-prod-bounds:eq}, respectively, we immediately obtain
\begin{align*}
\frac{\sigma_{\min}(V_A)}{\sigma_{\max}(V_B)} \sigma_{\min}(\widetilde C \circ R)
\le
\sigma_{\min}(X)
\le
\frac{\sigma_{\max}(V_A)}{\sigma_{\min}(V_B)} \sigma_{\min}(\widetilde C \circ R),\\
\frac{\sigma_{\min}(V_A)}{\sigma_{\max}(V_B)} \sigma_{\max}(\widetilde C \circ R)
\le
\sigma_{\max}(X) \le
\frac{\sigma_{\max}(V_A)}{\sigma_{\min}(V_B)} \sigma_{\max}(\widetilde C \circ R),
\end{align*}
which immediately yields~\eqref{eq:statement}.
\end{proof}

If $A$ and $B$ are normal, then $V_A$ and $V_B$ are unitary and the result simplifies further.
\begin{corollary}
\label{kappa-lyap-normal:cor}
If $A$ and $B$ are normal, then
\[
\sigma_{\min}(X) =
\sigma_{\min}(\widetilde C \circ R), \quad \sigma_{\max}(X) =
\sigma_{\max}(\widetilde C \circ R), \quad
\kappa_2(X) = \kappa_2(\widetilde C \circ R).
\]
\end{corollary}

In general, for every matrix $\widetilde C$ and $R$, we have~\cite[Thm.~5.5.1]{hojo91}
\begin{equation}
\label{max-sv-hadamard:eq}
\sigma_{\max}(\widetilde C\circ R)
\le
\sigma_{\max}(\widetilde C) \, \sigma_{\max}(R)
\end{equation}
although this bound may be a substantial overestimation.\footnote{See~\cite[Thm. 5.5.3]{hojo91} for smaller upper bounds, which may nevertheless be pessimistic.}

The following result gives a sufficient condition under which the solution of the Lyapunov equation $AX+XA^T=-C$, with $A$ symmetric positive definite and $C$ symmetric negative definite, is
well conditioned. In particular, both $A$ and $C$ being well conditioned is
sufficient to guarantee good conditioning of $X$. However, this condition is not
necessary: the solution can be well conditioned even when $A$ and $C$ are moderately
ill conditioned, as shown in \cref{spd-A-C-lyap:ex}.
\begin{theorem} \label{spd-A-C-lyap:thm}
If $A$ is symmetric positive definite and $C$ is symmetric negative definite, then the unique solution to the Lyapunov equation $AX+XA^T = -C$ is symmetric positive definite and satisfies
\begin{equation}
\label{spd-A-C-lyap-bound:eq}
1 \le \ \kappa_2(X)
\ \le \
2 \kappa(C) \, \lambda_{\max}(A)\, \max_{i} \sum_{j=1}^n \frac{1}{\lambda_i + \lambda_j}
\ \le \
n\, \kappa_2(A)\, \kappa_2(C),
\end{equation}
where $\lambda_i$ denotes the $i$-th eigenvalue of $A$.
\end{theorem}

\begin{proof}
    Let $A=:QDQ^T$ be the spectral decomposition of $A$. Then the Lyapunov equation $AX + XA^T = -C$ is equivalent to
    \begin{equation*}
        \Lambda Y + Y \Lambda = \widetilde C,\qquad \widetilde C = -Q^TCQ,\qquad Y = Q^TXQ,
    \end{equation*}
    where the solution $Y$ can be written as
    \begin{equation*}
        Y = \widetilde C \circ R, \qquad R_{ij} = \frac{\widetilde C_{ij}}{\lambda_i + \lambda_j}.
    \end{equation*}
    $C$ is symmetric negative definite, thus $\widetilde C$ is symmetric positive definite. Since $R$ is symmetric positive semidefinite with nonzero diagonal entries, $Y$ is symmetric positive definite by~\cite[Thm. 5.2.1]{hojo91}.
    Since $Q$ is orthogonal, $X$ is also symmetrix positive definite and
    \begin{equation*}
        \kappa_2(X) = \kappa_2(Y) = \frac{\lambda_{\max}(Y)}{\lambda_{\min}(Y)}.
    \end{equation*}
    Using the inequalities on~\cite[p.~312]{hojo91} and \cite[Thm.~5.3.4]{hojo91}, we can easily show that
    \begin{align*}
        \lambda_{\max}(Y)
        &= \norm{Y}_2
        = \norm{\widetilde C \circ R}_2 \le \norm{\widetilde C}_2 \norm{R}_2
        = \lambda_{\max}(C) \norm{R}_2,\\
        \lambda_{\min}(Y)
        &= \lambda_{\min}(\widetilde C \circ R)
        \ge \lambda_{\min}(\widetilde C) \min_i R_{ii}
        =\lambda_{\min}(\widetilde C) \min_i\frac{1}{2\lambda_i}
        = \frac{\lambda_{\min}(C)}{2\lambda_{\max}(A)}.
    \end{align*}
    and therefore that
    \begin{equation}\label{eq:partial-bound}
        \kappa_2(X) \le 2 \kappa_2(C) \lambda_{\max}(A) \norm{R}_2.
    \end{equation}
    To bound $\norm{R}_2$, note that $R$ is positive semidefinite and therefore
    \begin{equation}\label{eq:norm-R-bound}
        \norm{R}_{2}
        = \sqrt{\norm{R}_1 \norm{R}_{\infty}}
        = \norm{R}_{\infty}
        = \max_{1 \le i \le n}\sum_{j = 1}^{n} \frac{1}{\lambda_i + \lambda_j}
        \le \sum_{j=1}^n \frac{1}{2 \lambda_{\min}(A)} = \frac{n}{2 \lambda_{\min}(A)},
    \end{equation}
    where the second equality hold because $R$ is symmetric and the third because it is positive semidefinite.
    Combining~\eqref{eq:partial-bound} and~\eqref{eq:norm-R-bound} gives~\eqref{spd-A-C-lyap-bound:eq}.
\end{proof}

\subsection{Lower bounds in terms of Zolotarev numbers}
\label{zolotarev:sec}
We next use a result of Beckermann and Townsend~\cite{beto17, beto19} to develop lower bounds for $\kappa_2 (X)$ in terms of Zolotarev numbers when the coefficient matrices are normal, rather than merely diagonalizable.

\begin{theorem} \label{Zol-lower-bound-normal-thm}
Let $A \in \mathbb{C}^{m \times m}$ and $B\in \mathbb{C}^{n \times n}$ be normal matrices with $m \geq n$ and let $E$ and $F$ be complex sets such that the spectrum of $A$ is a subset of $E$ and the spectrum of $B$ is a subset of $F$. Suppose that the matrix $X\in \mathbb{C}^{m \times n}$, of rank $r$, satisfies
\begin{equation}
\label{sylv-decomposedRHS:eq}
AX-XB = MN^\ast, \qquad M \in \mathbb{C}^{m \times \nu}, \quad N \in \mathbb{C}^{n \times \nu},
\end{equation}
where $1 \le \nu \le n$ is an integer. Then,
for integers $0\le k \le \frac{n-1}{\nu}$, we have
\[
\frac{\sigma_1(X)}{\sigma_{1+\nu k}(X)} \ge \bigl(Z_k(E, F)\bigr)^{-1},
\]
where
\[
Z_k(E, F) := \inf_{\widetilde r \in \mathcal{R}_{k, k}} \frac{\sup_{z \in E} |\widetilde r(z)|}{\inf_{z \in F} |\widetilde r(z)|}
\]
is the Zolotarev number where $\mathcal{R}_{k, k}$ is the set of irreducible rational functions of the form $p(x)/q(x)$ and $p$ and $q$ are polynomials of degree at most $k$.
In particular,
\begin{equation}\label{Zol-lower-bound-normal-thm:eq}
\kappa_2(X) \ge \bigl(Z_\ell(E, F)\bigr)^{-1},\qquad \ell = \Bigl\lfloor \frac{r-1}{\nu} \Bigr\rfloor
\end{equation}
\end{theorem}

\begin{proof}
This is a direct consequence of~\cite[Thm.~2.1]{beto19}, if we take $\ell$ such that $r = \rank(X) \ge 1 + \nu \ell$. Hence, $\sigma_{r} \le \sigma_{1 + \nu \ell}$, which is then bounded from above by $Z_\ell(E, F) \sigma_{1}(X)$.
\end{proof}

The displacement rank of $X$ with respect to $A$ and $B$ is $\rank(AX-XB)$. Therefore, from~\eqref{sylv-decomposedRHS:eq}, $X$ has displacement rank at most $\nu$. The ratio $\ell = \lfloor \frac{r-1}{\nu}\rfloor$ is then the largest admissible rational degree for which $1+\nu \ell \le r$ hence $\sigma_{1+\nu \ell}$ is still among the potentially nonzero singular values of $X$.

Zolotarev numbers have been extensively studied in the literature; see~\cite{gptv15,nafr16} for instance. For certain choices of $E$ and $F$, in particular, explicit bounds on $Z_\ell(E, F)$ are known. Below, we list three such bounds, which apply when $E$ and $F$ being real symmetric intervals, general real intervals, and disks, respectively.

\begin{enumerate}[label=C\arabic*., ref=C\arabic*]
\item \label{it:c1} Let $0 < a < b < \infty$. Then
\begin{equation}
\label{Zol-bnd-ring-fun:eq}
\frac{4 \rho^{-2\ell}}{(1+\rho^{-4 \ell})^4} \le
Z_\ell\bigl( [-b, -a], [a, b] \bigr) \le
\frac{4 \rho^{-2\ell}}{(1+\rho^{-4\ell})^2}
\le 4 \rho^{-2\ell},
\end{equation}
with
\begin{equation*}
    \rho = \exp \Big(\frac{\pi^2}{2 \mu(a/b)} \Big),
\end{equation*}
where $\mu(\cdot)$ is the Gr\"otzsch ring function; see~\cite[Cor.~3.2]{beto19}. The simpler but looser bound is
\[
Z_\ell\bigl( [-b, -a], [a, b] \bigr) \le 4 \biggl[ \exp \Big( \frac{\pi^{2}}{2 \log (4b/a)} \Big) \biggr]^{-2\ell}.
\]

\item \label{it:c2} Consider the Zolotarev numbers $Z_\ell\big( [a, b], [c, d] \big)$, where either $b < c$ or $d<a$, so that $[a, b] \cap [c, d] = \emptyset$. Then,
\[
Z_\ell \bigl( [a, b], [c, d] \bigr)
\le
4 \biggl[ \exp \Bigl( \frac{\pi^{2}}{2 \mu (1/\alpha)} \Bigr) \biggr]^{-2\ell}
\le
4 \biggl[ \exp \Bigl( \frac{\pi^{2}}{2 \log (16 \gamma)} \Bigr) \biggr]^{-2\ell},
\]
with
\[
\alpha = -1 + 2\gamma + 2 \sqrt{\gamma^2 - \gamma}, \qquad \gamma = \frac{|c-a| |d-b|}{|c-b| |d-a|},
\]
in which $\gamma$ is the cross-ratio of the two spectra.

\item \label{it:c3} Explicit formulae are also available for certain configurations of $E$ and $F$ involving disks~\cite[sec.~3.3]{beto19}. Let either $E$ or $F$ be the closed disk centered at $c$ with radius $r_1$, and let the other set be the exterior of the concentric open disk of radius $r_2 > r_1$. Then,
\[
Z_\ell\bigl( \{ z \in \mathbb{C} \ :\ |z-c| \le r_1 \}, \{ z \in \mathbb{C}\ :\ |z-c| \ge r_2\} \bigr)
= \Bigl(\frac{r_1}{r_2}\Bigr)^\ell.
\]

In addition, if $E = \bigl\{ z \in \mathbb{C} \ :\ |z-\frac{a+b}{2}| \le \frac{b-a}{2} \bigr\}$ is a disk whose diameter is the real interval $[a, b]$ and $-E$ is the mirror image with diameter $[-b, -a]$, then
\[
Z_\ell(-E, E)
= \biggl( \frac{1 - \sqrt{a/b}}{1+\sqrt{a/b}} \biggr)^{2\ell}.
\]
The above formulae are applicable also when either $E$ or $F$ is a half-plane.
\end{enumerate}

We note that if $r \le \nu$, then $\ell = 0$, and for the three special cases~\ref{it:c1}--\ref{it:c3}, \cref{Zol-lower-bound-normal-thm} only gives the trivial lower bound $\kappa_2(X) \ge 1$.

On the other hand, in all three special cases, the lower bound on $\kappa_2(X)$ grows exponentially with $\ell$ between the rank $r$ of $X$ and its displacement rank $\nu$. Thus, when $r$ is (moderately) large {\em relative to} $\nu$, every solution $X$ satisfying the hypotheses is unavoidably ill-conditioned. This corresponds to problems arising in important applications in which the right-hand side $C$ has low rank; see~\cite{bes15, krto10, simo16} and the references therein.\footnote{As an extreme example, if $A$ is Hurwitz stable and $(A, b)$ is controllable, then $A X+X A^T = -b b^T$ has a symmetric positive definite solution. Thus, a rank-1 right-hand side may give a rank-$n$ solution. See~\cite[Thm.~4.15 and 4.18(a), and Prop.~4.27]{anto05} for instance.}

In particular, for \ref{it:c1}, we have
\begin{equation}
\label{kappa-Zol-bnd-ring-fun:eq}
\kappa_2(X) \ge \frac{1}{4} \rho^{2\ell}
\end{equation}
and the strength of the bound depends on the ratio $\lambda = b/a$ appearing in $\rho$, not merely on the fact that the intervals are disjoint.

Similarly, for \ref{it:c2}, we have
\[
\kappa_2(X) \ge \frac{1}{4} \biggl[ \exp \Big( \frac{\pi^{2} \ell}{\log (16 \gamma)} \Big) \biggr],
\]
showing that $X$ is ill-conditioned whenever $\ell$ is sufficiently large relative to
$\log(16\gamma)$ which depends on the cross-ratio of the two spectra.

Finally, for \ref{it:c3}, we see that
\[
\kappa_2(X) \ge \Bigl(\frac{r_2}{r_1}\Bigr)^\ell
\]
and so the solution is ill-conditioned, when $(\frac{r_2}{r_1})^\ell$ is large.

The following result establishes a connection between \cref{diagonalizable-sylv:prop} and \cref{Zol-lower-bound-normal-thm}.
\begin{remark}
If $A$ and $B$ are diagonalizable, but not necessarily normal, then following~\eqref{sylv-diagonal:eq}, the Sylvester equation~\eqref{sylv-decomposedRHS:eq} can be reformulated as
\[
D_A \widetilde X - \widetilde X D_B = \widetilde M \widetilde N^{\ast}, \qquad \widetilde M = V_A^{-1}M, \quad \widetilde N = V_B^{\ast} N,
\]
whose coefficient matrices $D_A$ and $D_B$ are normal. Hence, \cref{Zol-lower-bound-normal-thm} applies to this transformed equation. Moreover, the value of $\ell$ remains unchanged, since the displacement rank of $\widetilde X$ and $X$ coincide, as follows from $\rank(\widetilde M \widetilde N^{\ast}) = \rank(M N^{\ast})$. Then, combining the resulting bound, for instance in Case 1, with the first inequality in \cref{diagonalizable-sylv:prop}, we obtain
\[
\kappa_2(X) \ge \max \Bigg\{ \frac{\exp \Big( \frac{\pi^{2} \ell}{ \log (4b/a)} \Big)}{4 \kappa_2(V_A) \kappa_2(V_B)}, 1 \Bigg\}.
\]
Analogous lower bounds can be derived for Cases~2 and~3.
\end{remark}

\section{Bounds using integral representation} \label{integral-rep:sec}

\noindent A straightforward differentiation and integration argument shows that, if the expression
\[
X=\int_{0}^{\infty} e^{-tA}\,C\,e^{+tB}\,dt
\]
exists for all $C$, then it represents the unique solution of~\eqref{sylv:eq}; see~\cite[pp.~318--319]{high02} for instance.

Applying \cref{log-norm:lem} gives
\[
\|X\|_2 \le \int_{0}^{\infty} \|e^{-At}\|\,\|C\|\,\|e^{Bt}\|\,dt
\le \|C\|\ \int_{0}^{\infty} e^{t\mu(-A) + t\mu(B)} \,dt
\]
If $\mu(-A) + \mu(B) < 0$, then the integral converges to $\frac{-1}{\mu(-A) + \mu(B)}$
giving
\begin{equation}
\label{sol-norm-upper-bound-mu:eq}
\|X\|_2 \le \frac{\|C\|_2}{|\mu(-A) + \mu(B)|} \le \frac{\|C\|_2}{\sep_2(A,B)},
\end{equation}
where the last step follows directly from \cref{sol-log-sep:lem}.

Let $B=-A^T$, i.e., consider the Lypaunov equation~\eqref{lyap:eq-intro}. It follows that if $A$ is (Hurwitz) stable, i.e., all its eigenvalues have negative real parts and $C$ is symmetric positive definite, then
\begin{equation}
\label{eq:lyap-int}
X= \int_{0}^{\infty} e^{t A}\,C\,e^{t A^T}\,dt
\end{equation}
is the unique symmetric positive definite solution of~\eqref{lyap:eq-intro}.

In the case of \cref{amazing:ex},
\ignore{
\[
X=\int_{0}^{\infty} e^{\begin{pmatrix} 0 & -1\\ -1 & 0 \end{pmatrix}t}\,C\,e^{\begin{pmatrix} 0 & -2\\-2 & 0 \end{pmatrix}t}\,dt
\]
Note that
}
\[
\|e^{-tA}\| = e^t, \quad \mbox{ and } \quad \|e^{tB}\| = e^{2t}.
\]
Let $\epsilon=10^{-10}$. Then at $t=5$ and $t=10$, the norm of the integrand is $4.9 \times 10^{6}$ and $1.6 \times 10^{13}$, respectively. This could be a hint that we may need to account for the transient growth in the hump of matrix exponentials. Indeed, there is a connection between the problem of identifying conditions under which $\kappa(X)$ is modest, and the hump of matrix exponential. It is known that
\begin{equation}
\label{eq:kappa-lower-bound}
\kappa(X) \ge \| e^{t A} \|^2 \, e^{t /\|X\|}, \quad t \ge 0
\end{equation}
where all norms are spectral norms, and $X$ is the symmetric positive definite solution of the Lyapunov equation
\begin{equation}
\label{lyap-eye:eq}
AX + X A^T = -I.
\end{equation}
See~\cite{Godunov90},~\cite[Thm.~2]{vese97} and~\cite[p.~147]{trem09} for a proof.

Here is another lower bound on the condition number of the solution to the Lyapunov equation.
\begin{theorem}\label{thm:xu-trace} \cite{vese98,xu97}
Let $A\in \R^{n\times n}$ be Hurwitz stable, let $C$ be symmetric positive definite, and $X$ be the symmetric positive definite solution of $AX+XA^T = -C$. Then, for all $t\ge0$,
\begin{equation}
\label{kappa-lower-bound-trace:eq}
\kappa_2(X) \ge
        \frac{\Tr(e^{tA}\ e^{t A^T })}{\Tr\bigl(e^{-tC/\|X\|_2}\bigr)}.
\end{equation}
For the special case $C=I$, the following bounds hold
\begin{equation}
\label{sol-inverse-lower-bound-trace-special-case:eq}
\|X^{-1}\|_2 \ge
        \frac{\Tr(e^{tA}\ e^{t A^T })}{\Tr(X)}\, e^{t/\|X\|_2}.
\end{equation}
and
\begin{equation}
\label{kappa-lower-bound-trace-special-case:eq}
\kappa_2(X) \ge
        \frac{\|X\|_2}{\Tr(X)} \Tr(e^{tA}\ e^{t A^T })\, e^{t/\|X\|_2}.
\end{equation}
\end{theorem}

The lower bound~\eqref{kappa-lower-bound-trace:eq} is from~\cite[Thm. 4 \& Eq.(10)]{xu97}\footnote{where it is stated for $AX+XA^{\ast}=C$ with $C$ negative definite.}. The particular cases~\eqref{sol-inverse-lower-bound-trace-special-case:eq} and~\eqref{kappa-lower-bound-trace-special-case:eq} can be found in~\cite[Eq.~(9)]{vese98} and the latter is an improvement\footnote{in the sense that $\frac{1}{n} \le \frac{\|X\|_2}{\Tr(X)} = \frac{\lambda_{\max}(X)}{\sum_{i=1}^n \lambda_i(X)} \leq 1$ where the final inequality is strict when $n\ge 2$.} of
\begin{equation}
\label{kappa-lower-bound-trace-special-case-weaker:eq}
\kappa_2(X) \ge
        \frac{1}{n} \Tr(e^{tA}\ e^{t A^T })\, e^{t/\|X\|_2}.
\end{equation}
established in~\cite[Cor. 7]{xu97} for the special case $C=I$.

We now give a set of sufficient conditions under which a solution to the Lyapunov equation is well conditioned. For simplicity, we assume that $A$ and $C$ are real matrices so that $A^\ast = A^T$ which is convenient when dealing with possibly complex eigenvalues and eigenvectors of $A$. Our discussion relies on the Loewner order, denoted by $\preceq$, where $A \preceq B$ if $B - A$ is symmetric positive semi-definite. The following technical lemma shows that integration preserves Loewner order.

\begin{lemma}\label{lem:loew-int}
    Let $F,G:[0, \infty) \to \R^{n\times n}$ be continuous symmetric matrix-valued functions such that $F(t) \preceq G(t)$ for all $t \ge 0$.
    If the improper integrals
    \begin{equation*}
        \int_0^\infty F(t)\,dt,
        \qquad
        \int_0^\infty G(t)\,dt
    \end{equation*}
    exist entrywise, then
    \begin{equation}\label{eq:loewner-int}
        \int_0^\infty F(t)\,dt
        \preceq
        \int_0^\infty G(t)\,dt.
    \end{equation}
\end{lemma}
\begin{proof}
    For every $x \in \R^n$, we have
    \begin{equation*}
        x^T \bigl(G(t) - F(t)) x \ge 0,\qquad t \ge 0.
    \end{equation*}
    Therefore, the proper integral
    \begin{equation*}
        x^T\left(\int_0^\gamma (G(t) - F(t))\,dt\right)x
        =
        \int_0^\gamma x^T\bigl(G(t) - F(t)\bigr)x\,dt
        \ge 0.
    \end{equation*}
    exists for every $\gamma \ge 0$,
    and
    \begin{equation*}
        \int_0^\gamma \bigl(G(t) - F(t)\bigr)\,dt \succeq 0,\qquad \gamma \ge 0.
    \end{equation*}
    By linearity, the improper integral exists entrywise, and we have that
    \begin{equation*}
        \int_0^\infty \bigl(G(t) - F(t)\bigr)\,dt = \lim_{\gamma\to\infty}\int_0^\gamma \bigl(G(t) - F(t)\bigr)\,dt
    \end{equation*}
    Since the cone of positive semidefinite matrices is closed, the limit exists and is positive semidefinite, and therefore
    \begin{equation*}
        \int_0^\infty \bigl(G(t) - F(t)\bigr)\,dt \succeq 0.
    \end{equation*}
    Using again the linearity of the integral, we obtain~\eqref{eq:loewner-int}.
\end{proof}

Roughly speaking, the derivation of the following bound uses the integral representation~\eqref{eq:lyap-int} to express $X$ in terms of the exponentials of $A$ and $A^T$. If $A$ is strictly dissipative, then the smallest and largest singular values of these exponentials can be bounded in terms of the logarithmic norms. If, in addition, we assume that $C$ is symmetric positive definite, then the Loewner order allows us to carry these bounds through the integral representation, thereby yielding bounds on the extreme singular values of $X$ and hence on its condition number.

\begin{theorem} \label{Lyap-dissipativeA:thm}
Let $C$ be symmetric positive definite, and $A$ be strictly dissipative. Then, the unique solution $X$ to the Lyapunov equation $A X + X A^T = -C$
satisfies
\[
\kappa_2(X) \le \kappa_2(A+A^T)\, \kappa_2(C).
\]
\end{theorem}

\begin{proof}
Since $A$ is strictly dissipative, there exist constants
$0<\beta\le \gamma$ such that
\begin{equation*}
    -\gamma I_m \preceq \frac{A+A^T}{2} \preceq -\beta I_m.
\end{equation*}
Using the logarithmic norm, we can write this equivalently as
\begin{equation*}
    \mu(A)\le -\beta,
    \qquad
    \mu(-A)\le \gamma.
\end{equation*}
By \cref{log-norm:lem}, we therefore have
\begin{equation*}
    \norm{e^{tA}}_2 \le e^{-\beta t},
    \qquad
    \norm{e^{-tA}}_2 \le e^{\gamma t},
    \qquad t\ge 0.
\end{equation*}
Since $e^{tA}$ is nonsingular, we have that
\begin{equation*}
    \sigma_{\min}(e^{tA})
    =
    \frac{1}{\norm{e^{-tA}}_2}
    \ge e^{-\gamma t},
\end{equation*}
and therefore
\begin{equation}\label{eq:exp-bound}
    e^{-2\gamma t} I_m
    \preceq
    e^{tA}e^{tA^T}
    \preceq
    e^{-2\beta t} I_m,
    \qquad t\ge 0.
\end{equation}
$C$ is symmetric positive definite, thus $\lambda_{\min}(C) I_m \preceq C \preceq \lambda_{\max}(C) I_m$, and therefore
\begin{equation}\label{eq:sylv-integrand-bound}
    \lambda_{\min}(C)e^{tA}e^{tA^T}
    \preceq
    e^{tA} C e^{tA^T}
    \preceq
    \lambda_{\max}(C)e^{tA}e^{tA^T}.
\end{equation}
Combining~\eqref{eq:exp-bound} and~\eqref{eq:sylv-integrand-bound}, we obtain
\begin{equation}\label{eq:lowner-bounds}
    \lambda_{\min}(C)e^{-2\gamma t} I_m
    \preceq
    e^{tA} C e^{tA^T}
    \preceq
    \lambda_{\max}(C)e^{-2\beta t} I_m.
\end{equation}
Since $\norm{e^{tA}}_2\le e^{-\beta t}$, the integral \eqref{eq:lyap-int} is convergent, and $X$ satisfies the Lyapunov equation $AX+XA^T=-C$.
By~\cref{lem:loew-int}, we can integrate the Loewner bounds~\eqref{eq:lowner-bounds}, obtaining
\begin{equation*}
    \frac{\lambda_{\min}(C)}{2\gamma} I_m
    \preceq
    X
    \preceq
    \frac{\lambda_{\max}(C)}{2\beta} I_m,
\end{equation*}
which shows that
\begin{equation*}
    \lambda_{\min}(X)
    \ge
    \frac{\lambda_{\min}(C)}{2\gamma},
    \qquad
    \lambda_{\max}(X)
    \le
    \frac{\lambda_{\max}(C)}{2\beta},
\end{equation*}
and therefore that $X$ is positive definite.
Therefore,
\begin{equation}\label{eq:first-bound}
    \kappa_2(X)
    =
    \frac{\lambda_{\max}(X)}{\lambda_{\min}(X)}
    \le
    \frac{\gamma}{\beta}\kappa_2(C).
\end{equation}
Finally, the eigenvalues of $A+A^T$ lie in $[-2\gamma,-2\beta]$, whic implies that
\begin{equation}\label{eq:symm-part-cond-bound}
    \kappa_2(A+A^T)=\frac{\gamma}{\beta},
\end{equation}
Plugging~\eqref{eq:symm-part-cond-bound} into~\eqref{eq:first-bound} concludes the proof.
\end{proof}

Note that, in the special case $C=I$, the bounds reduce to \[
\|X\|_2 \le \frac{1}{2\beta},
\qquad
\|X^{-1}\|_2 \le 2\gamma,
\qquad
\kappa_2(X)\le \frac{\gamma}{\beta} = \kappa_2(A+A^T).
\]

\begin{remark}
We note that dissipativity of $A$ is a stronger assumption than Hurwitz stability. For example, $A = \begin{bmatrix}
-1 & M\\
0 & -1
\end{bmatrix}$ is Hurwitz stable, but not dissipative for $M > 2$ as its symmetric part
$\frac{1}{2} \begin{bmatrix}
-2 & M\\
M & -2
\end{bmatrix}$ has eigenvalues $-1 \pm \frac{M}{2}$. Then, with $C = I$, the solution is
$X = \frac{1}{4} \begin{bmatrix}
M^2+2 & M\\
M & 2
\end{bmatrix}$
with $\kappa_2(X) \approx M^2$ which could be huge.
\end{remark}

\section{Experiments} \label{sec:experiments}
The following example illustrates \cref{spd-A-C-lyap:thm}.

\begin{example}
\label{spd-A-C-lyap:ex}
We construct $10{,}000$ Lyapunov equations of the form $AX+XA^T = -C$ using randomly generated $3 \times 3$ matrices $A$ and $C$, where $A$ is symmetric positive definite and $C$ is symmetric negative definite. The matrix $A$ is generated as $A = Q D Q^T$ where $Q$ is the orthogonal factor in the QR decomposition of a random matrix with normally distributed entries and $D$ is a diagonal matrix with nonzero entries drawn uniformly from $[10^{-8}, 1]$. The matrix $C$ is generated in the same way and then multiplied by $-1$.

We then solve the resulting Lyapunov equations using MATLAB's \texttt{lyap} function. The left panel of \cref{spd-A-C-lyap:fig} shows a histogram of the 2-norm condition numbers of the computed solutions $X$. The median value of $\kappa_2(X)$ is $1.94 \times 10^4$, while the smallest and largest values are $2.97$ and $1.47 \times 10^{12}$, respectively.

\begin{figure}[t]
\centering

\center
\includegraphics[height=0.45\textwidth]{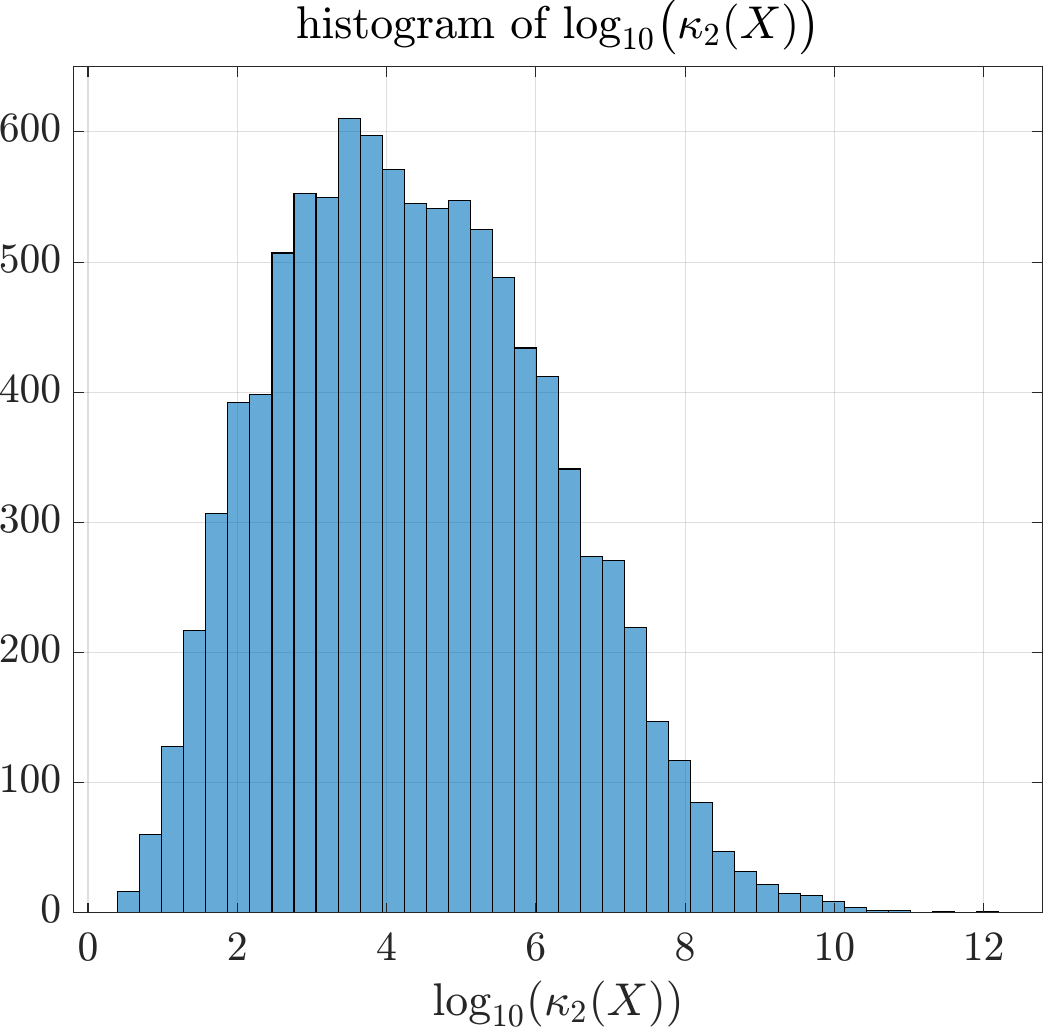}
\hfill
\includegraphics[height=0.45\textwidth]{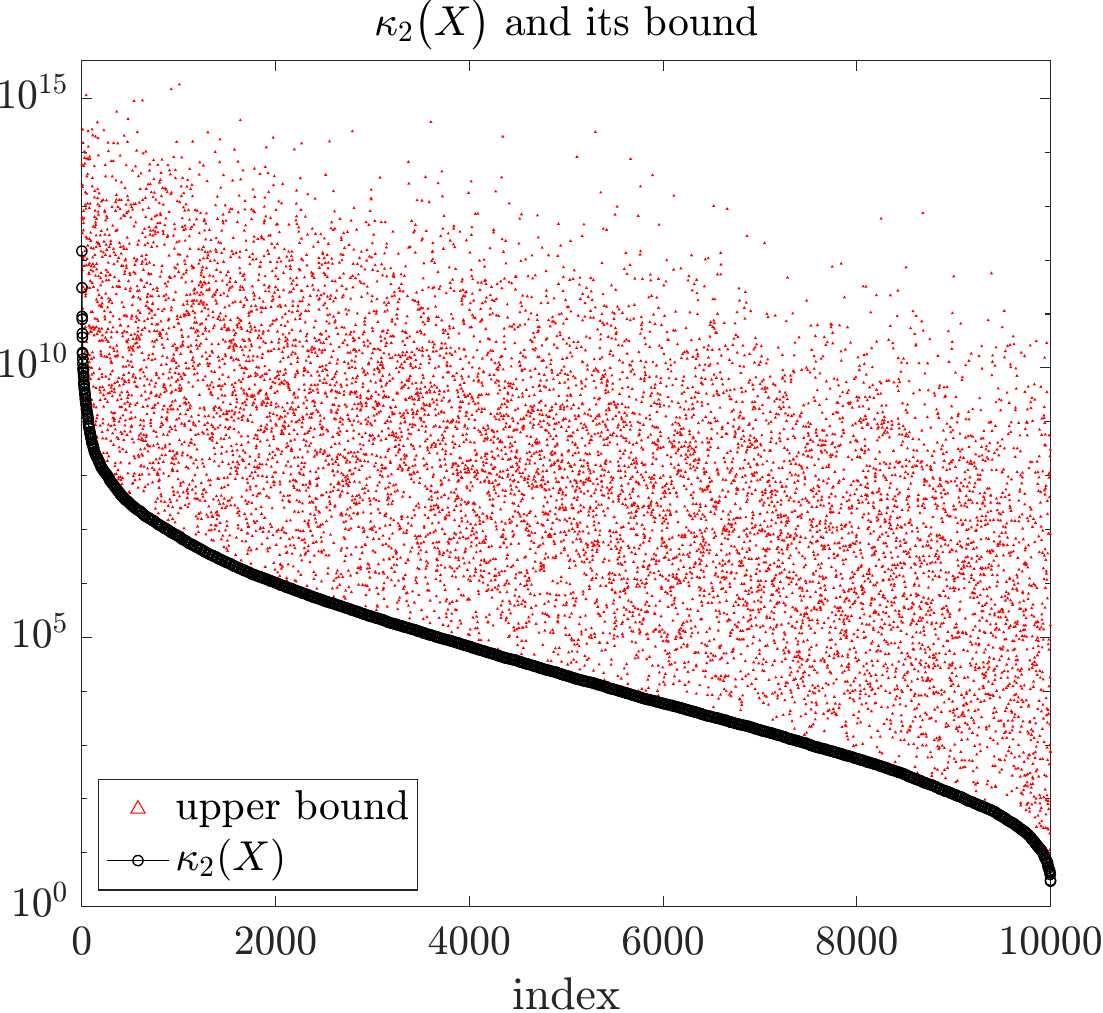}
\caption{Results for \cref{spd-A-C-lyap:ex}: frequency distribution of the condition numbers of the solutions (left), and comparison of the condition numbers with the bound (right).
}
\label{spd-A-C-lyap:fig}
 \end{figure}

In \cref{largest-kappaX-example:tab}, we report the details of the matrix whose solution achieves the larges spectral condition number, namely  $\kappa_2(X) \approx 1.47 \times 10^{12}$. The smallest singular value of the resulting matrix $\mathcal{A}$ is $9.88 \times 10^{-8}$.
\begin{table}[t]
\centering
\caption{Data for the test matrix producing the largest observed value of $\kappa_2(X)$.}
\label{largest-kappaX-example:tab}
\begin{tabularx}{\linewidth}{Xcc}
\toprule
Matrix & Eigenvalues & 2-norm condition number \\
\midrule
$A$ &
$4.94 \times 10^{-8},\, 3.60 \times 10^{-3},\, 1.69 \times 10^{-1} $ &
$3.42 \times 10^{6\hphantom{0}}$ \\

$C$ &
$-1.02 \times 10^{-8},\, -8.61 \times 10^{-8},\, -1.77 \times 10^{-1} $ &
$1.73 \times 10^{7\hphantom{0}}$ \\

$X$ &
$9.86 \times 10^{-7},\, 4.63,\, 1.45 \times 10^{6}$ &
$1.47 \times 10^{12}$ \\
\bottomrule
\end{tabularx}
\end{table}

\begin{table}[!ht]
\centering
\caption{Data for the test matrix producing the largest observed gap between $\kappa_2(X)$ and the bounds in \cref{spd-A-C-lyap:thm}.}
\label{largest-kappaX-gap-example:tab}
\begin{tabularx}{\linewidth}{Xcc}
\toprule
Matrix & Eigenvalues & 2-norm condition number \\
\midrule
$A$ &
$8.24 \times 10^{-7},\, 3.72 \times 10^{-4},\, 2.43 \times 10^{-1}$ &
$2.95 \times 10^{5}$ \\

$C$ &
$-1.74 \times 10^{-8},\, -3.69 \times 10^{-7},\, -4.39 \times 10^{-1}$ &
$2.51 \times 10^{7}$ \\

$X$ &
$8.57 \times 10^{-1},\, 2.55 \times 10^{1},\, 1.76 \times 10^{2} $ &
$2.06 \times 10^{2}$ \\
\bottomrule
\end{tabularx}
\end{table}

The bounds in \cref{spd-A-C-lyap:thm} are a priori bounds and may overestimate the actual condition numbers.\footnote{In principle, sharp estimates may be obtained, but this comes at the cost of requiring knowledge of all eigenvectors of $A$ and working directly with $\widetilde C \circ R$ which is the transformed solution $\widetilde X$; see \cref{kappa-lyap-normal:cor}. Such estimates are therefore more naturally viewed as a posteriori rather than a priori.} In the right panel of \cref{spd-A-C-lyap:fig}, we display the condition number and relative bound for 200 of the sampled matrices, ordered by condition number. There are examples in which $A$ and $C$ are moderately ill-conditioned, while $X$ remains well conditioned. In our numerical experiments, the largest observed gap between the actual value of $\kappa(X)$ and the bounds in \cref{spd-A-C-lyap:thm} occurs for the test problem reported in \cref{largest-kappaX-gap-example:tab}. We attribute this gap to the overestimation caused by inequalities such as~\eqref{max-sv-hadamard:eq}.

\end{example}

\subsection{Illustration of the Zolotarev lower bounds}
We now present two examples to illustrate the lower bounds on $\kappa_2(X)$ derived using Zolotarev numbers in \cref{zolotarev:sec}.

\begin{example}
\label{zolotarev:ex}
We construct a Sylvester equation with $m=n=50$. The coefficients $A = Q_A D_A Q_A^T$ and $B = Q_B D_B Q_B^T$ are symmetric, with $Q_A$ and $Q_B$ random orthogonal matrices and $D_A$ and $D_B$ diagonal matrices with nonzero entries logarithmically spaced in $[a, b]$ and $[-b, -a]$, respectively.\footnote{For results on the sharpness of Zolotarev bounds and the possible gap between such bounds for a discrete set of nodes and its convex hull, see \cite[pp.~395-396]{begr10} and \cite[p.~1570]{mpr18}, and the references therein.} We set $a = 0.5$ and $b = 10$ and randomly perturb the nodes in $[-b, -a]$ so that they are not simply reflections\footnote{In the absence of this perturbation, the problem corresponds to a Lyapunov equation, and the results are qualitatively similar.} of those in $[a, b]$. We choose the right-hand side matrix $C = M N^T$, where $M$ and $N$ are
length-$n$ vectors with
entries
drawn independently from a normal distribution.

\begin{figure}[t]
\centering
\hspace{40pt}

\center
\includegraphics[height=0.41\textwidth]{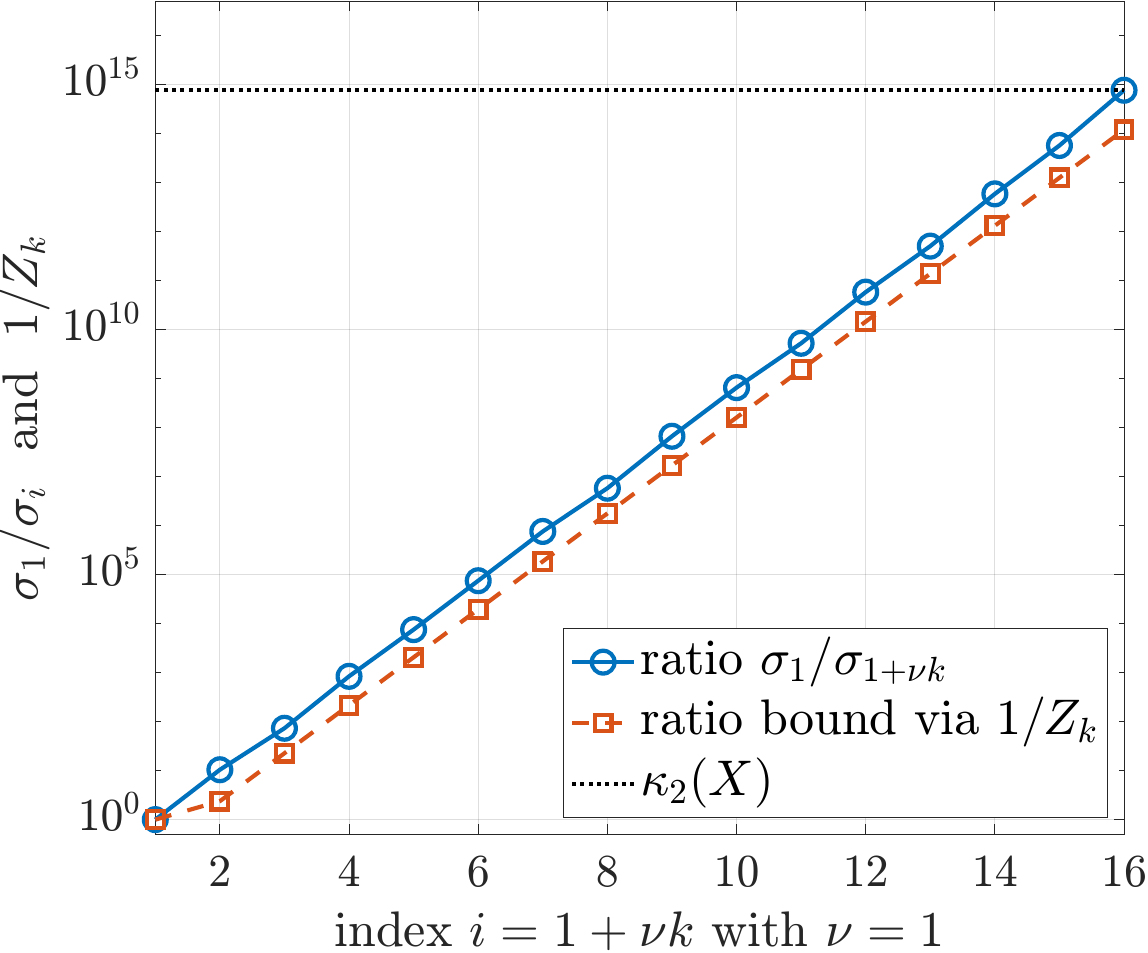}
\hfill
\includegraphics[height=0.41\textwidth]{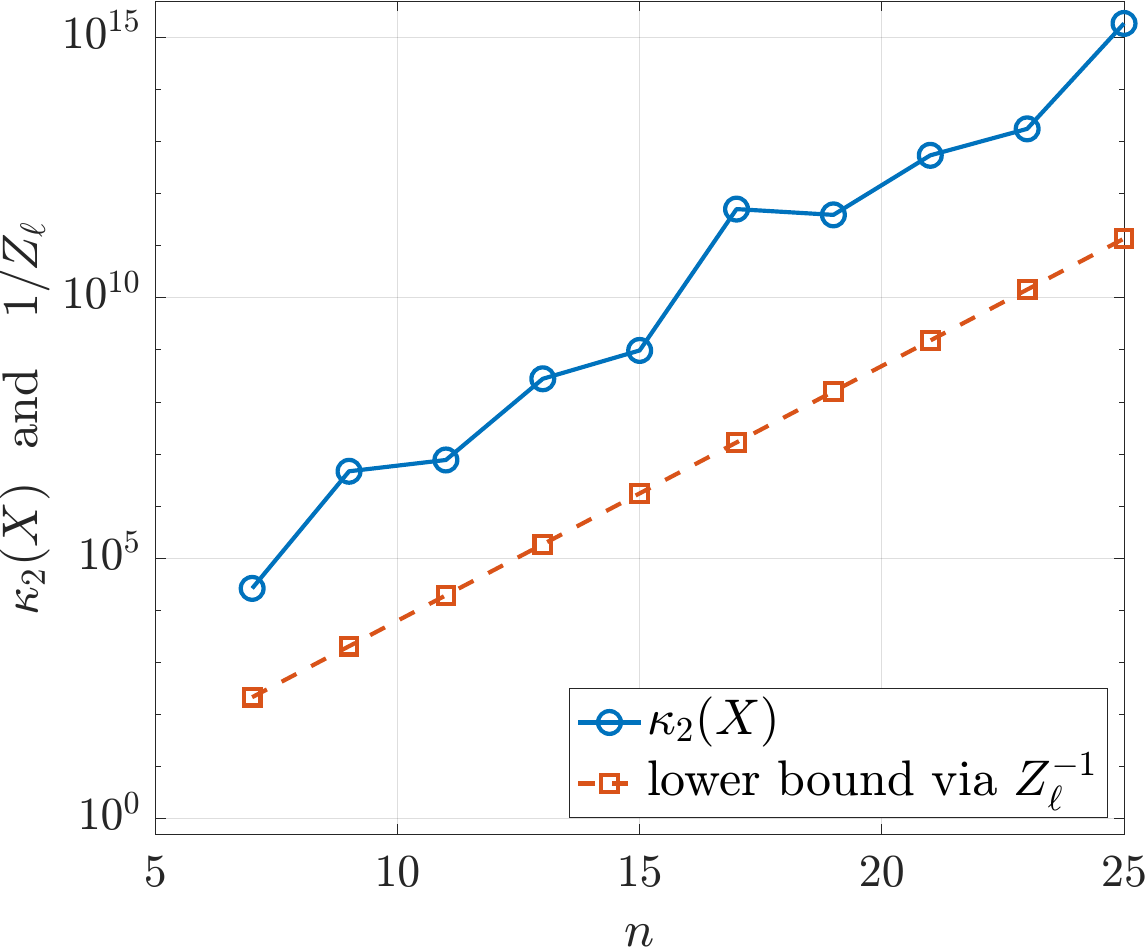}
\caption{Left: Ratio of $\sigma_1(X)$ to other singular values versus their lower bounds via Zolotarev numbers. Here, $n=50$ and the numerical rank of $X$ is $16$. See \cref{zolotarev:ex}. Right: Growth of $\kappa_2(X)$ and its lower bound $Z_{\ell}^{-1}$ across ten problems with varying size $n$. See~\cref{zolotarev2:ex}.}
\label{zolotarev:fig}
\end{figure}

We then compute the ratios
\[
\frac{\sigma_1(X)}{\sigma_{i}(X)}, \qquad i = 1+k, \quad k=0,1,\ldots,\ell, \qquad \ell = \biggl\lfloor \frac{\widehat r-1}{\nu}\biggr\rfloor,
\]
where $\widehat r$ denotes the numerical rank of the solution $X$. In this example, $\widehat r=16$ and $\ell=15$. The left panel of \cref{zolotarev:fig} shows these ratios together with the corresponding bounds from \cref{Zol-lower-bound-normal-thm}. As the index $1+k$ increases, the singular-value ratios approach $\kappa_2(X)$, indicated by the horizontal line. The Zolotarev lower bounds closely track these ratios and ultimately the condition number as well. In particular, for $k=5,10$, and $15$, we obtain the following:
\[
\begin{alignedat}{3}
7.6 \times 10^{4}
&\approx \frac{\sigma_1(X)}{\sigma_{6}(X)}
&\; {}\ge{} \;& Z_{5}(E,F)^{-1}
&\; {}\ge{} \;& 1.9 \times 10^{4}, \\[1mm]
5.8 \times 10^{9}
&\approx \frac{\sigma_1(X)}{\sigma_{11}(X)}
&\; {}\ge{} \;& Z_{10}(E,F)^{-1}
&\; {}\ge{} \;& 1.5 \times 10^{9}, \\[1mm]
4.8 \times 10^{14}
&\approx \frac{\sigma_1(X)}{\sigma_{16}(X)}
&\; {}\ge{} \;& Z_{15}(E,F)^{-1}
&\; {}\ge{} \;& 1.2 \times 10^{14}.
\end{alignedat}
\]
The lower bounds for $Z_k(E,F)^{-1}$ displayed above are computed as $\frac{1}{4}\rho^{2k}$. This expression is obtained from the upper bound in~\eqref{Zol-bnd-ring-fun:eq}, with $\ell$ replaced by the more general index $k$, where $\rho = \exp\big(\frac{\pi^2}{2\mu(\lambda)}\big)$.
The Gr\"otzsch function $\mu(\lambda)$ is evaluated at
$\lambda=\frac{a}{b}=0.05$ with the help of the MATLAB function \verb+ellipke+; see~\cite[Eq.~(3.2)]{beto19}.

\end{example}

The next example focuses solely on bounds on the condition number of the solution of a few different Lyapunov equations of varying size.
\begin{example}
\label{zolotarev2:ex}
We construct ten Lyapunov equations of size $n \times n$ with $n = 7, 9, 11, \dots, 25$ and with increasingly ill-conditioned solutions. For each problem, the matrix $A$ is generated as in the previous example, and we set $B= -A^T$ and $C = M N^{T}$
of rank $\nu =  2$. We then compute $\kappa_2(X)$ and the corresponding Zolotarev lower bound in~\eqref{kappa-Zol-bnd-ring-fun:eq}. The right panel of \cref{zolotarev:fig} shows that the Zolotarev bounds can closely track the growth of the condition number of the solution to such equations.

\end{example}

\subsection{When can a small separation lead to an ill-conditioned solution?} \label{experiments-sep-cond:sec}
The following examples examine simple Sylvester equations with the aim of exploring conditions under which a small separation may or may not contribute to the ill-conditioning of the solution. We begin with the diagonal setting.

\begin{example}
\label{diagonal_sylv_n_3:ex}

Let $A = \diag(1,4,7)$ and $B = \diag(10,1-\delta,20)$. Then, $\sep(A,B)$ is attained at the smallest spectral gap $\alpha_1-\beta_2=\delta$. Since $A$ and $B$ are diagonal, the solution is given by $X = C \circ R$, where $R$ is the Cauchy matrix~\eqref{rij:eq}
\[
    \renewcommand{\arraystretch}{1.4}
    R =
    \begin{bmatrix}
        -\frac{1}{9} &  \frac{1}{\delta}   &   -\frac{1}{19}\\
        -\frac{1}{6} &  \frac{1}{3+\delta} &   -\frac{1}{16}\\
        -\frac{1}{3} & \frac{1}{6+\delta}  &   -\frac{1}{13}
    \end{bmatrix}.
\]
We test 30 logarithmically spaced values of $\delta$ in the interval $[10^{-6},1]$ and plot $\kappa_2(X)$ in the left panel of \cref{sep-cond:fig}. Let $E_{12}$ denote the zero matrix unless its $(1,2)$ entry set to one and $\mathbf{1}$ denote the matrix of all-ones. The right-hand side $C$ is chosen as one of the following four matrices:
\[
I+E_{12} =\begin{bmatrix}
    1   &  1  &   0\\
     0 &    1  &   0\\
     0  &   0  &   1
\end{bmatrix}, \quad
\mathbf{1} = \begin{bmatrix}
    1   &  1  &   1\\
     1 &    1  &   1\\
     1  &   1  &   1
\end{bmatrix},
\quad
\mathbf{1} - E_{12} = \begin{bmatrix}
    1   &  0  &   1\\
     1 &    1  &   1\\
     1  &   1  &   1
\end{bmatrix},
\quad
I = \begin{bmatrix}
    1   &  0  &   0\\
     0 &    1  &   0\\
     0  &   0  &   1
\end{bmatrix}
\]
the key distinction being that $C_{12}$ is nonzero in the first two and zero in the latter two.
\begin{figure}[t]
\centering

\center
\includegraphics[height=0.44\textwidth]{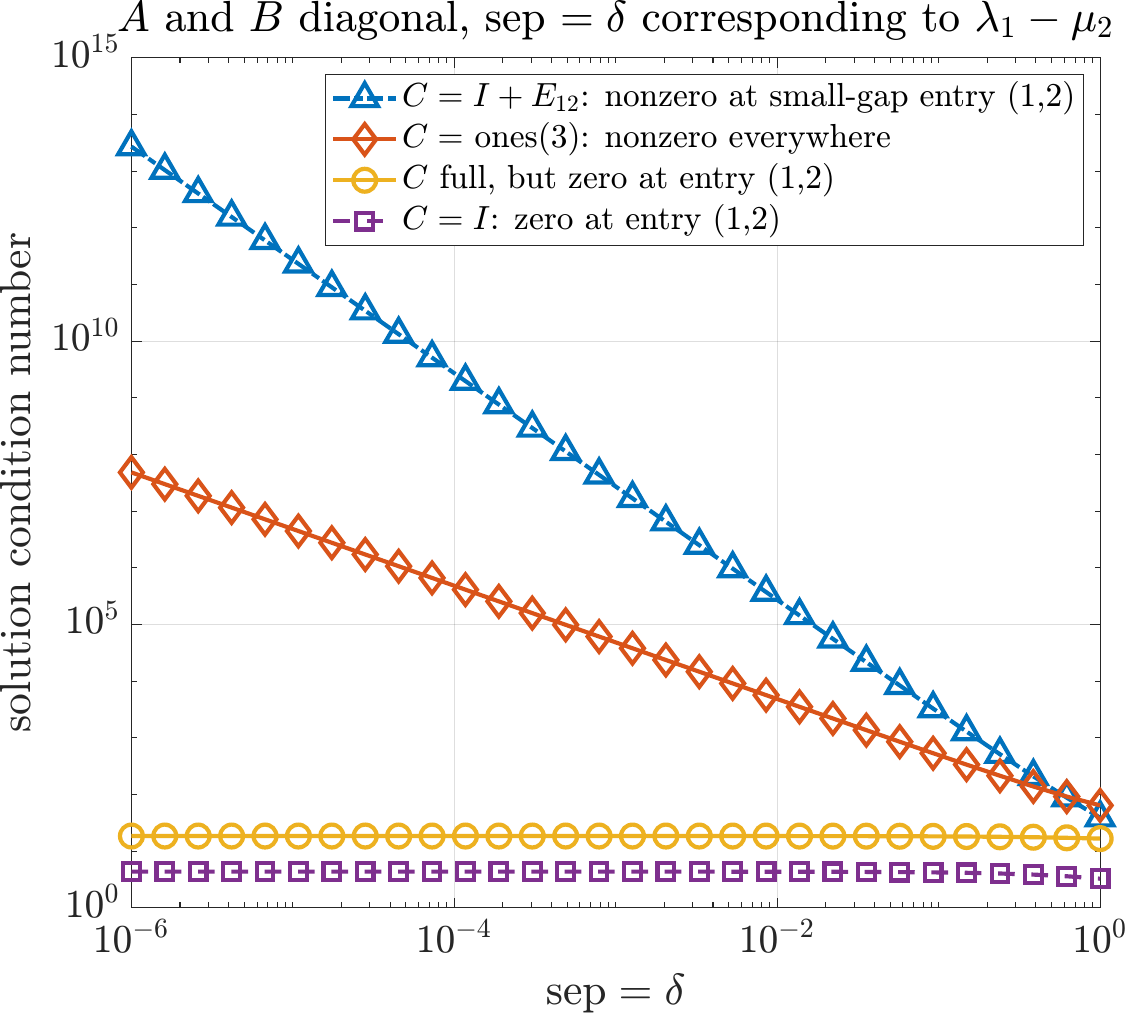}
\hfill
\includegraphics[height=0.44\textwidth]{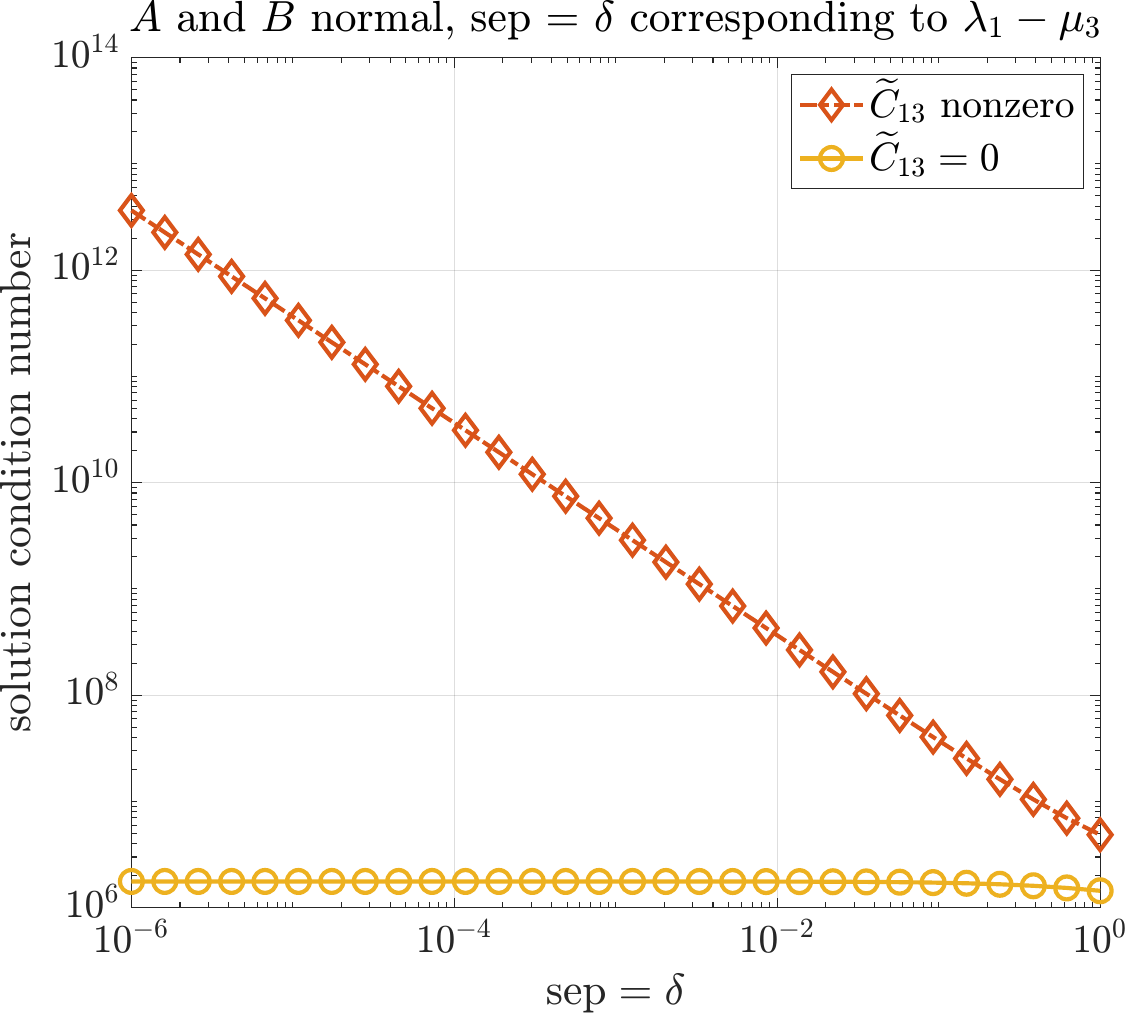}
\caption{Left (\cref{diagonal_sylv_n_3:ex}): $\kappa_2(X)$ versus separation with diagonal $A, B$ and four different right-hand side matrices $C$. Here, $n=3$ and the smallest gap between the spectra of $A$ and $B$ occurs at the first eigenvalue $\lambda_1(A) = 1$ and the second eigenvalue $\mu_2(B) = 1-\delta$. Right (\cref{normal_sylv_n_5:ex}): $\kappa_2(X)$ versus separation with full normal $A, B$ and two different right-hand sides $C$. Here, $n=5$ and the smallest gap between the spectra of $A$ and $B$ occurs at the first eigenvalue $\lambda_1(A) = 1$ and the third eigenvalue $\mu_3(B) = 1-\delta$.
}
\label{sep-cond:fig}
 \end{figure}

We observe that, with the first $C$, $\kappa_2(X)$ is of order $\delta^{-2}$, whereas in the second case $\kappa_2(X)$ is of order $\delta^{-1}$. In both cases, $\sigma_{\max}(X)$ is the same; the difference is that $\sigma_{\min}(X)$ is larger for the second choice of $C$, which reduces $\kappa_2(X)$ from order $\delta^{-2}$ to $\delta^{-1}$. Therefore, the rate at which the solution becomes ill conditioned depends on how the right-hand side $C$ loads the small gap between the two spectra. We also observe that $X$ remains well conditioned when the $(1,2)$ entry of $C$ is zero.
\end{example}

\begin{example}
\label{normal_sylv_n_5:ex}
We now revisit the previous observations in the case where $A$ and $B$ are full normal matrices, hence their eigenvector matrices do not affect the conditioning of the solution $X$. See \cref{kappa-lyap-normal:cor}.

We set $\lambda = [1, 3, 5, 7, 9]$, $\mu = [20,\ 22,\ 1-\delta,\ 24,\ 26]$ and take $A= U_A \diag(\lambda) U_A^T$ and $B= U_B \diag(\mu) U_B^T$ where $U_A$ and $U_B$ are $5\times 5$ randomly generated orthogonal matrices. With this construction, the smallest gap between the two spectra, and hence the separation of $A$ and $B$, is equal to $\delta$ and this gap is attained by the eigenvalue pair ($\lambda_1, \mu_3)$.

We test 30 logarithmically spaced values of $\delta$ in the interval $[10^{-6},1]$ and plot $\kappa_2(X)$ in
the right panel of \cref{sep-cond:fig}. We consider two choices for the transformed right-hand side matrix $\widetilde C$. The first is the all-ones matrix, which in particular keeps the (1,3) entry of the Cauchy matrix $R$ `active' in $\widetilde X$. The second is again the all-ones matrix, except that its (1,3) entry is set to zero. This removes the contribution associated with the separation and thereby keeps the transformed solution $\widetilde X$ controlled. In each case, the actual right-hand side $C$ is $C = U_A \widetilde C U_B^T$; see~\eqref{Ctilde-def:eq} and~\eqref{Xtilde-def:eq}. In both cases, the right-hand side matrices $C$ are fully dense and have visually indistinguishable patterns in the original basis. For reference, the two matrices corresponding to $\delta = 10^{-6}$ are shown below, rounded to three significant digits
\[
\resizebox{\linewidth}{!}{$\displaystyle
\begin{bmatrix}
 -0.239 & -0.112 & -1.29 & -0.568 & -1.02\\
  -0.00985 &  0.0850 &  1.86  & 0.659 &  1.33\\
   0.0250 &  0.0833  & 1.66 &  0.605  & 1.20\\
  -0.525 & -0.217 & -2.22 & -1.028 & -1.78\\
   0.703 &  0.169  & 0.365  & 0.447 &  0.521
\end{bmatrix},
\begin{bmatrix}
    -0.314 &-0.128 & -1.29&  -0.601 & -1.04\\
   0.450  & 0.184 &  1.85  & 0.862 &  1.49\\
   0.403 &  0.165  & 1.66 &  0.771 &  1.34\\
  -0.541 & -0.220 & -2.22  &-1.03 & -1.79\\
   0.0927 &  0.0379  & 0.381&   0.178 &  0.307
\end{bmatrix}.
$}
\]

We note that the first matrix $C$ corresponds to $\widetilde C_{13}=0$ and gives a solution $X$ with condition number $1.8 \times 10^6$, whereas the second matrix $C$ gives a solution with $\kappa_2(X) \approx 3.7 \times 10^{12}$. The norms of the two matrices $C$ are comparable: the first has norm $4.83$, while the second has norm $5$. From the construction of $\widetilde C$, and since multiplication by orthogonal matrices preserves rank, it is clear that the first matrix $C$ above has rank two, whereas the second has rank one. Both matrices $C$ are extremely ill-conditioned numerically, with computed condition numbers $8.7 \times 10^{18}$ and $3.3 \times 10^{18}$, respectively.

Our results show that the observations from \cref{diagonal_sylv_n_3:ex} essentially carry over to normal matrices. In this setting, however, whether a small separation contributes to the ill-conditioning of the solution is determined by the entries of $\widetilde C$, that is, by the representation of $C$ in the eigenbases of $A$ and $B$, rather than by the literal entries of $C$ in the original basis.
\end{example}

\bibliographystyle{plain}
\bibliography{refs}

\end{document}